\documentclass[11pt, oneside]{article}   	
\usepackage{geometry}                		
\usepackage{graphicx}				
								
\usepackage[all]{xy}
\CompileMatrices

\usepackage{amssymb}
\usepackage{amsmath}
\usepackage{stmaryrd}
\usepackage{amsthm}
\usepackage{tikz-cd}
\usepackage{extarrows}
\usepackage[mathscr]{euscript}
\usepackage{mathrsfs} 
\usepackage{verbatim}
\usepackage[OT2,T1]{fontenc}
\DeclareSymbolFont{cyrletters}{OT2}{wncyr}{m}{n}
\DeclareMathSymbol{\Sha}{\mathalpha}{cyrletters}{"58}
\DeclareMathSymbol{\Che}{\mathalpha}{cyrletters}{"51}

\usepackage{calligra}
\usepackage{mathrsfs}
\newcommand{\calHom}{\mathscr{H}\mathit{om}}

\newcommand{\Ga}{{\mathbf{G}}_{\rm{a}}}
\newcommand{\Gm}{{\mathbf{G}}_{\rm{m}}}

\DeclareMathOperator{\Pic}{Pic}

\DeclareMathOperator{\Ext}{Ext}
\DeclareMathOperator{\Hom}{Hom}

\DeclareMathOperator{\Spec}{Spec}
\DeclareMathOperator{\coker}{coker}

\newcommand*{\Z}{\ensuremath{\mathbf{Z}}}                        
\newcommand*{\Q}{\ensuremath{\mathbf{Q}}}                     
\newcommand*{\F}{\ensuremath{\mathbf{F}}}                        
\newcommand*{\C}{\ensuremath{\mathbf{C}}}                        
\newcommand*{\A}{\ensuremath{\mathbf{A}}}                        
\renewcommand*{\P}{\ensuremath{\mathbf{P}}}                        
\newcommand*{\calO}{\mathcal{O}}                                  

\usepackage{rotating}

\usepackage{bm}

\numberwithin{equation}{section}

\newtheorem{theorem}{Theorem}[section]

\newtheorem{lemma}[theorem]{Lemma}
\newtheorem{proposition}[theorem]{Proposition}

\theoremstyle{definition}

\theoremstyle{remark}
  \newtheorem{remark}[theorem]{Remark}

\usepackage[OT2,T1]{fontenc}

\tikzset{commutative diagrams/.cd,
mysymbol/.style={start anchor=center,end anchor=center,draw=none}
}

\usepackage{stmaryrd}
\usepackage{hyperref}

\title{\textbf{PATHOLOGICAL BEHAVIOR OF ARITHMETIC INVARIANTS OF UNIPOTENT GROUPS}}
\author{Zev Rosengarten \thanks{While completing this work, the author was supported by an ARCS Scholar Award and by a Ric Weiland Graduate Fellowship. \newline
}}

\date{}
\begin{document}
\maketitle

\begin{abstract}
We show that all of the nice behavior for Tamagawa numbers, Tate-Shafarevich sets, and other arithmetic invariants of pseudo-reductive groups over global function fields proved in \cite{rospred} fails in general for non-commutative unipotent groups. We also give some positive results which show that Tamagawa numbers do exhibit some reasonable behavior for arbitrary connected linear algebraic groups over global function fields.
\end{abstract}

\tableofcontents{}

\section{Introduction}

\subsection{Basic notions}

Let $G$ be a connected linear algebraic group over a global field $k$. One of the major arithmetic quantities with which this paper is concerned is the so-called {\em Tamagawa number} of $G$, denoted $\tau(G)$. This is the volume of a certain adelic coset space associated to $G$ with respect to a certain canonical measure, the Tamagawa measure, and it contains important arithmetic information about $G$. For more details, see \cite[Chap. I]{oesterle} and \cite[\S 1.1]{rospred}.

Also important for us will be the Tate-Shafarevich set of $G$, defined as
\[
\Sha^1(k, G) := \ker\left({\rm{H}}^1(k, G) \longrightarrow \prod_v {\rm{H}}^1(k_v, G)\right),
\]
the set of $G$-torsors over $k$ that have points everywhere locally (so $\Sha^1(k, G)$ measures the failure of the Hasse principle for $G$-torsors). We also denote the Tate-Shafarevich set by $\Sha^1(G)$ or simply $\Sha(G)$. Note that one may also define $\Sha(G)$ by the formula
\[
\Sha^1(k, G) := \ker\left({\rm{H}}^1(k, G) \longrightarrow {\rm{H}}^1(\A, G)\right)
\]
due to \cite[Prop.\,1.5]{rospred}.

The Tate-Shafarevich set of affine group schemes of finite type over global fields is finite. This finiteness is due to Borel and Serre over number fields (where one may easily reduce to the reductive case), Harder and Oesterl\'e in the reductive and solvable cases respectively over global function fields, and Conrad in the general case over global function fields; see \cite[\S 1.3]{conrad} and the references therein for more details. (The analogous finiteness of Tate-Shafarevich sets for abelian varieties is a major open problem.)

Modulo some results that were unknown at the time but have since been proven, Sansuc \cite[Th.\,10.1]{sansuc} obtained an elegant formula for Tamagawa numbers of connected reductive groups. He showed that for such $G$ one has
\begin{equation}
\label{sansucformula}
\tau(G) = \frac{\# \Pic(G)}{\# \Sha(G)}.
\end{equation}
Over number fields, this is the end of the story, at least for Tamagawa numbers of linear algebraic groups: one easily deduces that Sansuc's formula holds for all connected linear algebraic groups over number fields. The reason is that over a number field (or, more generally, over any perfect field), every linear algebraic group is an extension of a reductive group by a split unipotent group (i.e., a unipotent group admitting a filtration with successive quotients isomorphic to the additive group $\Ga$). Over imperfect fields, however, such as global function fields, this fails completely, and in fact, Sansuc's formula (\ref{sansucformula}) fails to hold in general even for forms of $\Ga$ over global function fields.

Nevertheless, one may obtain a suitable replacement for (\ref{sansucformula}) for a large class of groups, namely groups that are either commutative or pseudo-reductive. Recall that a connected linear algebraic group over a field $k$ is said to be {\em pseudo-reductive} if its $k$-unipotent radical $\mathscr{R}_{u, k}(G)$ -- which is defined to be the maximal smooth connected normal unipotent $k$-subgroup of $G$ -- is trivial. (Reductivity means that the same holds over $\overline{k}$.) Over perfect fields, reductivity and pseudo-reductivity agree, because Galois descent implies that the $\overline{k}$-unipotent radical of $G$ descends all the way down to $k$, but over imperfect fields there are many examples of groups that are reductive but not pseudo-reductive.

In order to formulate a replacement for (\ref{sansucformula}), we introduce a subgroup of $\Pic(G)$ which keeps track of the group structure on $G$, namely, for a smooth connected group scheme $G$ over a field $k$, let
\begin{equation}
\label{extdefn}
\Ext^1(G, \Gm) := \{ \mathscr{L} \in \Pic(G) \mid m^*\mathscr{L} \simeq \pi_1^*\mathscr{L} \otimes \pi_2^*\mathscr{L}\},
\end{equation}
where $m, \pi_i: G \times G \rightarrow G$ ($i = 1, 2$) are the multiplication and projection maps respectively. Thus, $\Ext^1(G, \Gm)$ is the group of line bundles on $G$ that are universally translation-invariant modulo line bundles on the base. The reason for the notation $\Ext^1(G, \Gm)$ is that any extension of $G$ by $\Gm$ is in particular a $\Gm$-torsor over $G$, hence we get a homomorphism $\Ext^1_{{\rm{Yon}}}(G, \Gm) \rightarrow \Pic(G)$, where the ``Yoneda Ext'' group $\Ext^1_{{\rm{Yon}}}(G, \Gm)$ is the set of $k$-isomorphism classes of extensions of $G$ by $\Gm$ made into a group via Baer sum. (Any such extension $E$ is automatically represented by a smooth connected affine $k$-group that is a central extension of $G$ by $\Gm$: the automorphism scheme of $\Gm$ is \'etale, hence the $k$-group map $E \rightarrow {\rm{Aut}}_{\Gm/k}$ induced by conjugation is constant.) This induces an isomorphism $\Ext^1_{{\rm{Yon}}}(G, \Gm) \xrightarrow{\sim} \Ext^1(G, \Gm)$. (This is essentially \cite[Thm.\,4.12]{colliot-thelene}, though the result there is stated only for $G$ of multiplicative type. The proof is the same in general, using Chevalley's Unit Theorem.)

When $G$ is commutative, the notation $\Ext^1(G, \Gm)$ may also be used to denote the derived-functor Ext in the category of fppf abelian sheaves on $\Spec(k)$. This latter Ext group is canonically isomorphic to the group $\Ext^1(G, \Gm)$ defined above \cite[Prop.\,4.3]{rospic}, so there is no ambiguity in the notation. Another nice property of the group $\Ext^1(G, \Gm)$ is that it is finite for any connected linear algebraic group $G$ over a {\em global} field $k$ \cite[Thm.\,1.1]{rospic}. This result is truly arithmetic in nature, as it fails over every local function field and over every imperfect separably closed field \cite[Prop.\,5.9]{rospic}. Let us also remark that if $k$ is a perfect field then the inclusion $\Ext^1(G, \Gm) \subset \Pic(G)$ is an equality for any connected linear algebraic $k$-group $G$ (see the paragraph preceding the statement of Theorem 1.5 in \cite{rospic}). In particular, this equality holds when $k$ is a number field.

\subsection{Positive results for commutative and pseudo-reductive groups}
\label{posresults}

Now we recall the main results of \cite{rospred}. One of these is the following generalization of Sansuc's formula (\ref{sansucformula}) to pseudo-reductive (and commutative) groups:

\begin{theorem}$($\cite[Thm.\,1.1]{rospred}$)$
\label{tamagawaformula}
Let $G$ be a connected linear algebraic group over a global field $k$. Assume that $G$ is either commutative or pseudo-reductive. Then
\[
\tau(G) = \frac{\# \Ext^1(G, \Gm)}{\# \Sha(G)}.
\]
\end{theorem}

\begin{remark}
\label{bsd}
Theorem \ref{tamagawaformula} is the analogue for linear algebraic groups of the Birch and Swinnerton-Dyer Conjecture; see the discussion after Theorem 1.1 in \cite{rospred}. 
\end{remark}

In \cite{rospred}, it is also shown that Tamagawa numbers and Tate-Shafarevich sets behave well with respect to inner twisting. Recall that if $G$ is a group scheme over a field $k$, then an {\em inner form} of $G$ is a $k$-form $G'$ of $G$ that is in the image of the map ${\rm{H}}^1(k, G/Z_G) \rightarrow {\rm{H}}^1(k, {\rm{Aut}}_{G/k})$, where $Z_G$ is the center of $G$, the last set classifies $k$-forms of $G$, and the map on cohomology sets is induced by the map $G/Z_G \rightarrow {\rm{Aut}}_{G/k}$ which sends $g$ to conjugation by $G$. We also say that $G'$ is obtained from $G$ by inner twisting. If $G$ is smooth, then so is $G/Z_G$, hence an inner form of $G$ is a $k_s/k$-form. \cite{rospred} proves the following result:

\begin{theorem}
\label{innerinvariance}$($\cite[Thm.\,1.4]{rospred}$)$
Let $G$ be a pseudo-reductive group over a global function field $k$. Then $\tau(G)$ and $\# \Sha(G)$ are invariant under inner twisting. That is, if $G'$ is an inner form of $G$, then $\tau(G) = \tau(G')$ and $\# \Sha(G) = \# \Sha(G')$.
\end{theorem}

\begin{remark}
\label{fibersoftheta}
Here is a consequence of Theorem \ref{innerinvariance}. Consider the map 
\begin{equation}
\label{thetamap}
\theta_G: {\rm{H}}^1(k, G) \longrightarrow \coprod_v {\rm{H}}^1(k_v, G).
\end{equation}
By definition, $\Sha(G) = \ker(\theta_G)$, but what about the other fibers? For $x \in {\rm{H}}^1(k, G)$, the fiber $\theta_G^{-1}(\theta_G(x))$ containing $x$ is in bijection with the set $\Sha(G_x)$, where $G_x$ is the $k$-form of $G$ obtained by taking the image of $x$ under the map ${\rm{H}}^1(k, G) \rightarrow {\rm{H}}^1(k, G/Z_G) \rightarrow {\rm{H}}^1(k, {\rm{Aut}}_{G/k})$. In particular, $G_x$ is an inner form of $G$! Theorem \ref{innerinvariance} therefore implies that when $G$ is pseudo-reductive, all of the nonempty fibers of $\theta_G$ have the same size (though that theorem is strictly stronger than this, since the map ${\rm{H}}^1(k, G) \rightarrow {\rm{H}}^1(k, G/Z_G)$ need not be surjective). 
\end{remark}

In light of Remark \ref{fibersoftheta}, it is natural to ask for which $\alpha \in {\rm{H}}^1(\A, G) = \coprod_v {\rm{H}}^1(k_v, G)$ the fiber $\theta_G^{-1}(\alpha)$ is nonempty. In order to answer this question, we note that for any smooth connected group scheme $G$ over a global field $k$, we have a complex of pointed sets
\begin{equation}
\label{H^1Extcomplex}
{\rm{H}}^1(k, G) \longrightarrow {\rm{H}}^1(\A, G) \longrightarrow \Ext^1(G, \Gm)^*
\end{equation}
which we will now define. The map ${\rm{H}}^1(k, G) \rightarrow {\rm{H}}^1(\A, G)$ is induced by the diagonal inclusion $k \hookrightarrow \A$. To define the second map, consider an extension
\begin{equation}
\label{extensioneqn1}
1 \longrightarrow \Gm \longrightarrow E \longrightarrow G \longrightarrow 1
\end{equation}
(which is automatically central since connectedness of $G$ implies connectedness of $E$) and an element $\alpha \in {\rm{H}}^1(\A, G)$. We obtain for each place $v$ of $k$ the element $\alpha_v \in {\rm{H}}^1(k_v, G)$; $\alpha_v$ is trivial for all but finitely many $v$ by \cite[Prop.\,1.5]{rospred}. Due to the centrality of the extension (\ref{extensioneqn1}), we get a connecting map ${\rm{H}}^1(k_v, G) \rightarrow {\rm{H}}^2(k_v, \Gm)$, this ${\rm{H}}^2$ being $\Q/\Z$, $\frac{1}{2}\Z/\Z$, or $0$, depending on whether $k_v$ is non-archimedean, $\mathbf{R}$, or $\C$ (by taking local Brauer invariants). 

Thus we get for each place $v$ of $k$ an element of $\Q/\Z$, all but finitely many of which are $0$. Adding these up produces an element of $\Q/\Z$. This procedure yields a map ${\rm{H}}^1(\A, G) \times \Ext^1(G, \Gm) \rightarrow \Q/\Z$. One may check that this map is additive in the second argument, hence induces a map of pointed sets ${\rm{H}}^1(\A, G) \rightarrow \Ext^1(G, \Gm)^*$. This defines the second map in (\ref{H^1Extcomplex}). Further, the image of any element of ${\rm{H}}^1(k, G)$ under this map is $0$, since the sum of the local invariants of a global Brauer class is $0$ by class field theory. 

\begin{remark}
\label{compatiblewithtateduality}
Before continuing, we note that the complex (\ref{H^1Extcomplex}) is compatible with global Tate duality in the following sense. Given a commutative 
affine group scheme $G$ of finite type over a field $k$, let $$\widehat{G} := \calHom(G, \Gm)$$ denote the fppf $\Gm$-dual sheaf. We have a functorial (in $G$ and $k$) exact sequence 
\begin{equation}
\label{H^1(G^)=kerExt}
0 \longrightarrow {\rm{H}}^1(k, \widehat{G}) \longrightarrow \Ext_{{\rm{cent}}}(G, \Gm) \longrightarrow \Ext_{{\rm{cent}}}(G_{\overline{k}}, (\Gm)_{\overline{k}})
\end{equation}
defined as follows. (When $G$ is disconnected, we have to specify that we are only considering central extensions of $G$ by $\Gm$, as arbitrary such extensions need not be central.) 

A central extension $E$ of $G$ by $\Gm$ splits fppf locally over $k$ if and only if it splits over $\overline{k}$, by the Nullstellensatz and standard spreading-out arguments. Thus, 
\[
\ker \left(\Ext_{{\rm{cent}}}(G, \Gm) \rightarrow \Ext_{{\rm{cent}}}(G_{\overline{k}}, (\Gm)_{\overline{k}})\right)
\]
consists of those extensions that split fppf locally, i.e., the fppf forms of the trivial extension $E = \Gm \times G$ (with the obvious extension structure). But one easily checks that the automorphism functor of the trivial extension (as an extension of $G$ by $\Gm$) is $\widehat{G}$. It follows that the above kernel is canonically (up to a universal choice of sign) isomorphic to ${\rm{H}}^1(k, \widehat{G})$. 

When $k$ is a global field, we have a complex
\begin{equation}
\label{H^1exactsequence1}
{\rm{H}}^1(k, G) \longrightarrow {\rm{H}}^1(\A, G) \longrightarrow {\rm{H}}^1(k, \widehat{G})^*,
\end{equation}
where the first map is induced by the diagonal inclusion $k \hookrightarrow \A$, and the second by cupping everywhere locally and adding the invariants. That is, given $\alpha \in {\rm{H}}^1(\A, G)$ and $\beta \in {\rm{H}}^1(k, \widehat{G})$, we have for each place $v$ of $k$ the cup product pairing 
\[
{\rm{H}}^1(k_v, G) \times {\rm{H}}^1(k_v, \widehat{G}) \rightarrow {\rm{H}}^2(k_v, \Gm) \xrightarrow{\sim} \Q/\Z,
\]
where the last map is once again the Brauer invariant. Thus, by taking the cup product of $\alpha_v$ and $\beta_v$ for each $v$ and then adding the results, we obtain the second map in (\ref{H^1exactsequence1}) above. (By the compatibility of cohomology with direct limits of rings \cite[Prop.\,D.0.1]{rostateduality}, both $\alpha_v$ and $\beta_v$ lift to ${\rm{H}}^1(\calO_v, \cdot)$ for all but finitely many $v$, hence their cup product lands in ${\rm{H}}^2(\calO_v, \Gm) = 0$, so the sum above contains only finitely many nonzero terms.) The sequence (\ref{H^1exactsequence1}) is a complex, once again because the sum of the local invariants of a global Brauer class is $0$. Part of the statement of global Tate duality is that the sequence (\ref{H^1exactsequence1}) is exact for any affine commutative $k$-group $G$ of finite type \cite[Thm.\,1.2.8]{rostateduality}. 

At any rate, the point we would like to make here is that for connected commutative affine $k$-group schemes $G$ of finite type,  the two complexes (\ref{H^1Extcomplex}) and (\ref{H^1exactsequence1}) are compatible via the first map in (\ref{H^1(G^)=kerExt}) (with the appropriate universal choice of sign). This compatibility may be checked by using the description of cup product (on ${\rm{H}}^1$) in terms of \v{C}ech cohomology.
\end{remark}

The following result tells us which fibers of the map $\theta_G$  in (\ref{thetamap}) are nonempty:

\begin{theorem}$($\cite[Thm.\,1.6]{rospred}$)$
\label{pseudoredcomplexexact}
Let $k$ be a global function field, and let $G$ be a connected linear algebraic $k$-group that is either commutative or pseudo-reductive. Then the complex
\[
{\rm{H}}^1(k, G) \longrightarrow {\rm{H}}^1(\A, G) \longrightarrow \Ext^1(G, \Gm)^*
\]
defined in $(\ref{H^1Extcomplex})$ is an exact sequence of pointed sets. That is, the kernel of the second map is the image of the first.
\end{theorem}

\subsection{Pathologies with unipotent groups}

Given the results of \S \ref{posresults}, it is natural to ask to what extent these results may be extended beyond the commutative and pseudo-reductive settings; that is, may they be generalized to arbitrary connected linear algebraic groups? Over number fields, the general case follows easily from the reductive case, because over such fields every connected linear algebraic group is an extension of a reductive group by a {\em split} unipotent group. The analogous statement over function fields is that every connected linear algebraic group is an extension of a pseudo-reductive group by a unipotent group. The problem is that this unipotent group need not be split. In fact, unipotent groups over imperfect fields can be extremely complicated. {\em The purpose of the present work is to show that all of the main results of \cite{rospred} discussed in \S \ref{posresults} fail even for general wound non-commutative 2-dimensional unipotent groups.} (For a discussion of woundness, see the beginning of \S \ref{gabbergps}.) We now discuss each of these failures in turn.

First, Theorem \ref{tamagawaformula} fails.

\begin{theorem}
\label{tamformulafails}
For every prime number $p$, there is a global function field $k$ of characteristic $p$ and a wound non-commutative $2$-dimensional unipotent group $U$ over $k$ for which Theorem $\ref{tamagawaformula}$ fails to hold. That is,
\[
\tau(U) \neq \frac{\# \Ext^1(U, \Gm)}{\# \Sha(U)}.
\]
\end{theorem}

For a more precise description of the groups $U$ in Theorem \ref{tamformulafails}, see Theorem \ref{tamfails}. Since Theorem \ref{tamagawaformula} is the linear algebraic group version of the Birch and Swinnerton-Dyer Conjecture (Remark \ref{bsd}), Theorem \ref{tamformulafails} suggests that Tamagawa numbers of unipotent groups exhibit rather odd behavior in general. With regard to their behavior under inner twisting (see Theorem \ref{innerinvariance}), the situation is even worse.

\begin{theorem}
\label{innerinvariancefailsforU}
Over every global function field $k$, there exist wound non-commutative $2$-dimensional unipotent groups $U$ having the following property: for every $\epsilon, M > 0$, there exists an inner form $U'$ of $U$ such that $\tau(U') < \epsilon$ and $\# \Sha(U') > M$.
\end{theorem}

\begin{remark}
As in Remark \ref{fibersoftheta}, Theorem \ref{innerinvariancefailsforU} has implications for the fibers of the map
\[
\theta_U: {\rm{H}}^1(k, U) \rightarrow \coprod_v {\rm{H}}^1(k_v, G).
\]
Namely, as discussed in that remark, for $x \in {\rm{H}}^1(k, U)$, the fiber $\theta_U^{-1}(\theta(x))$ is in bijection with the Tate-Shafarevich set $\Sha(U_x)$ of $U$ twisted by $x$. But we claim that twists of $U$ by elements of ${\rm{H}}^1(k, U)$ are the same as inner twists of $U$. {\em The Tate-Shafarevich aspect of Theorem \ref{innerinvariancefailsforU} therefore exactly says that the map $\theta_U$ has arbitrarily large fibers.}

To prove the claim, we first note that any twist of $U$ by an element of ${\rm{H}}^1(k, U)$ is an inner form of $U$ because the map ${\rm{H}}^1(k, U) \rightarrow {\rm{H}}^1(k, {\rm{Aut}}_{U/k})$ factors through ${\rm{H}}^1(k, U) \rightarrow {\rm{H}}^1(k, U/Z_U)$. On the other hand, we claim that the map ${\rm{H}}^1(k, U) \rightarrow {\rm{H}}^1(k, U/Z_U)$ is surjective. Since $Z_U \subset U$ is central, in order to prove this it suffices to show that ${\rm{H}}^2(k, Z_U) = 0$. In fact, this ${\rm{H}}^2$-vanishing holds for any commutative unipotent $k$-group, as one may show by reducing to the smooth connected case \cite[Prop.\,2.5.4(i)]{rostateduality}, $\alpha_p$, and finite \'etale commutative groups of $p$-power order (for which the fppf cohomology agrees with the \'etale by \cite[Thm.\,11.7]{briii}, hence the desired vanishing is \cite[Ch.\,II, \S2.2, Prop.\,3]{serre}).
\end{remark}

Despite the pathological behavior described in Theorem \ref{innerinvariancefailsforU}, Tamagawa numbers of connected linear algebraic groups over global function fields do exhibit some reasonable behaviors. First, although Tamagawa numbers in general vary over the collection of inner forms of a given $G$, they only vary by powers of $p$:

\begin{theorem}
\label{tautwistpowerofpthm}
If $G$ is a connected linear algebraic group over a global function field of characteristic $p$, and $G'$ is an inner form of $G$, then $\tau(G)/\tau(G') \in p^\Z$.
\end{theorem}

Second, Theorem \ref{innerinvariancefailsforU} naturally leads one to ask whether the unboundedness goes in the reverse direction. That is, clearly $\# \Sha(U')$ cannot be made arbitrarily small (it is bounded below by $1$), but can $\tau(U')$ be made arbitrarily large? The answer is no, as the following result shows.

\begin{theorem}
\label{tauboundedinnerthm}
Let $G$ be a connected linear algebraic group over a global function field. Then there is a constant $M$ $($depending on $G$$)$ such that $\tau(G') < M$ for all inner forms $G'$ of $G$.
\end{theorem}

Next we come to Theorem \ref{pseudoredcomplexexact}, which describes the fibers of the map $\theta_G$. This result, too, fails in general for non-commutative unipotent groups:

\begin{theorem}
\label{complexnotexact}
For every prime $p$ there exists a global function field $k$ of characteristic $p$ and a wound non-commutative $2$-dimensional unipotent $k$-group $U$ such that the complex of pointed sets
\[
{\rm{H}}^1(k, U) \longrightarrow {\rm{H}}^1(\A, U) \longrightarrow \Ext^1(U, \Gm)^*
\]
fails to be exact.
\end{theorem}

For a more precise description of the groups described in Theorem \ref{complexnotexact}, see Theorem \ref{complexnotexactU}.

Let us now summarize the contents of this paper. In \S \ref{gabbergps} we describe some wound non-commutative 2-dimensional unipotent groups constructed by Gabber which will serve as our source of counterexamples for the various pathologies described in this introduction. In \S \ref{sectionpathologies1}, we prove Theorems \ref{tamformulafails} and \ref{complexnotexact} by using Gabber's groups to construct explicit counterexamples to the unipotent analogues of Theorems \ref{tamagawaformula} and \ref{pseudoredcomplexexact}. In \S \ref{sectionpathologies2}, we prove Theorem \ref{innerinvariancefailsforU} by showing that all of the groups constructed by Gabber over global function fields provide examples of the pathological behavior described in that result. Finally, in \S \ref{innertwistpos} we prove Theorems \ref{tautwistpowerofpthm} and \ref{tauboundedinnerthm}. Each section of this paper is independent of the others, with the exception that \S\S \ref{sectionpathologies1} and \ref{sectionpathologies2} depend upon the constructions in \S \ref{gabbergps}.

\subsection{Notation and conventions}

Throughout this paper, $k$ denotes a field and, when it appears, $p$ denotes a prime number equal to the characteristic of $k$.

A {\em linear algebraic group} over $k$ is a smooth affine $k$-group scheme. When $k$ is a global field, $k_v$ denotes the completion of $k$ at a place $v$, $\calO_v$ the ring of integers of $k_v$ when $v$ is non-archimedean, and $\A_k$ (or just $\A$ when there can be no confusion) denotes the ring of adeles of $k$.

For any affine $k$-group scheme $G$ of finite type, we define $\widehat{G}$ to be the functor on $k$-algebras given by
$$\widehat{G}(A) = \Hom_{A-{\rm{gp}}}(G_A, \Gm)$$
(so $\widehat{G}(k)$ is the group of $k$-homomorphisms $G \to \Gm$).

We must also make some remarks about the behavior of cohomology in exact sequences. Given an affine group scheme $G$ of finite type over a field $k$, one may define ${\rm{H}}^1(k, G)$ as the set of fppf $G$-torsors over Spec$(k)$ up to isomorphism; this is a pointed set, and if $G$ is commutative then it is even an abelian group. When $G$ is commutative, one may also define the higher cohomology groups. Note that the affineness of $G$ implies that the torsor sheaves classified by this ${\rm{H}}^1$-set are all representable due to the effectivity of fppf descent for affine schemes.

If $G$ is smooth over $k$, then the \'etale and fppf $G$-torsors agree, so we may define ${\rm{H}}^1(k, G)$ to be the set of isomorphism classes of \'etale or fppf $G$-torsors over Spec$(k)$. The \'etale and fppf cohomology groups ${\rm{H}}^i(k, G)$ agree for all $i$ when $G$ is smooth and commutative \cite[Thm.\,11.7]{briii}. When $G$ is smooth, therefore, all of the defined cohomology groups may be defined in terms of Galois cohomology by using the language of cocycles and coboundaries. This is the language used in \cite{serre}.

On the other hand, for some purposes it is necessary to work with cohomology over general base schemes, or even fppf cohomology over fields, in which case the language of Galois cocycles is insufficient. One may sometimes replace these with \v{C}ech cohomology, but it is also useful to develop the entire theory in a more intrinsic manner, using the language of torsors.

Given an exact sequence
\begin{equation}
\label{asampleseq}
1 \longrightarrow G' \longrightarrow G \longrightarrow G'' \longrightarrow 1
\end{equation}
of smooth affine $k$-group schemes, one may compute the associated cohomology exact sequence using Galois cohomology. In this language, much of the familiar formalism of long exact sequences which comes out of (\ref{asampleseq}) when the groups are commutative remains true in the non-commutative setting. This is discussed in \cite[Chap.\,I, \S 5]{serre}. 

Much of the discussion in \cite{serre} is done in the language of torsors as well as Galois cocycles, though not all of it. We will at various points throughout this work require results for cohomology sets of the form ${\rm{H}}^1(\A_k, G)$, where $G$ is a smooth connected affine group scheme over a global field $k$. Strictly speaking, the results in \cite{serre} do not apply to these sets. There are, however, two ways around this.
The first is to simply invoke \cite[Prop.\,1.5]{rospred} to reduce assertions for adelic cohomology to the case of fields, where one may use Galois cohomology and apply the results in \cite{serre} directly. The second approach is to state and prove all of the results of \cite[Chap.\,I, \S 5]{serre} in the more general context of torsors over an arbitrary base rather than merely Galois (i.e., \'etale) cohomology over fields. This is essentially done in \cite[Appendix B]{conrad}. 

Strictly speaking, \cite[Appendix B]{conrad} only treats the case of fields, but all of the techniques and arguments used there for deriving properties of the long exact sequence associated to the short exact sequence (\ref{asampleseq}) work over a general base ring for affine groups, since affineness ensures the effectivity of all descent datum, which is necessary if one wishes to ensure that all fppf sheaf torsors are in fact representable by schemes; see \cite[\S III.4.1, Prop.\,1.9]{demazure}. (If one does not care about such representability, and is satisfied with working just with sheaf torsors, then even this assumption is unnecessary.)

Throughout this work, we will refer to \cite{serre} and invoke \cite[Prop.\,1.5]{rospred}, but we wanted to make the reader aware of the more general results essentially proved in \cite[Appendix B]{conrad} which allow one to work directly over the adele ring rather than invoking this ``trick''.

\subsection{Acknowledgements}

It is a pleasure to thank my advisor, Brian Conrad, who worked tirelessly to provide me with many helpful suggestions and improvements on both the presentation and the arguments throughout this paper.

\section{Gabber's unipotent groups $U_a$}
\label{gabbergps}

In this section we describe some groups constructed by Gabber (for the purpose of giving examples of wound non-commutative unipotent groups, which is actually somewhat tricky). These are the groups that we shall use in this paper to give counterexamples to the results in \S \ref{posresults} in the (wound non-commutative) unipotent setting. These yield examples of 2-dimensional wound non-commutative unipotent groups over every imperfect field.

Before giving the constructions, we briefly recall the notion of woundness for unipotent groups. A {\em wound}  (or {\em $k$-wound}) unipotent group over a field $k$ is a smooth connected unipotent $k$-group scheme $U$ such that any map  of $k$-schemes $\A^1_k \rightarrow U$ from the affine line to $U$ is the constant map to some $u \in U(k)$. This is equivalent to saying that $U$ does not contain a copy of $\Ga$ \cite[Def.\,B.2.1, Cor.\,B.2.6]{cgp}. Woundness is insensitive to separable field extension. That is, if $U$ is a smooth connected unipotent $k$-group, and $K/k$ is a (not necessarily algebraic) separable field extension, then $U$ is $k$-wound if and only if $U_K$ is $K$-wound \cite[Prop.\,B.3.2]{cgp}.
Clearly, smooth connected $k$-subgroups of wound unipotent $k$-groups are still wound. So are extensions of wound groups by other wound groups, as may be checked by using the formulation in terms of containing a copy of $\Ga$. Quotients of wound groups, however, need not be wound (even quotients by smooth connected subgroups). Over a perfect field, all smooth connected unipotent groups are split \cite[Thm.\,15.4(iii)]{borelalggroups}, hence no nontrivial wound unipotent groups exist over such fields. Over imperfect fields, however, there are many. We will give some examples below (see (\ref{defofVW})).

Now we recall Gabber's construction of wound non-commutative unipotent groups over any imperfect field \cite[Ex.\,2.10]{conrad3}. Let $k$ be an imperfect field of characteristic $p$, and let $a \in k - k^p$. Consider the smooth connected wound $1$-dimensional unipotent groups $V_a, W_a$ defined by
\begin{equation}\label{defofVW}
V_a := \{ X = X^{p^2} + aY^{p^2} \} \subset \Ga^2,\,\,\,
W_a := \{X = -(X^p + aY^p)\} \subset \Ga^2.
\end{equation}
One checks that these groups become isomorphic to $\Ga$ over $k(a^{1/p^2})$, $k(a^{1/p})$, respectively.
We claim that they are wound. Indeed, it suffices to show that they are not $k$-isomorphic to $\Ga$. We will in fact show that they are not isomorphic to $\A^1_k$ as $k$-schemes. In order to do this, it suffices to show that for the unique regular compactifications $\overline{V}_a, \overline{W}_a$ of the smooth curves $V_a, W_a$, the complements $\overline{V}_a - V_a, \overline{W}_a - W_a$ do not consist of a single rational point. (In fact, this method of checking woundness is completely general; for a smooth connected $1$-dimensional unipotent group $G$ over any field $k$, the regular compactification $\overline{G}$ of $G$ always consists of a single point that becomes rational over some finite purely inseparable extension of $k$, and $G$ is wound if and only if this point is not rational over $k$ \cite[Prop.\,5.3]{rospic}.)
But one easily checks that $\overline{V}_a, \overline{W}_a$ are given by the projectivizations of the equations in (\ref{defofVW}) defining $V_a, W_a$:
\[
\overline{V}_a := \{ XZ^{p^2-1} = X^{p^2} + aY^{p^2} \} \subset \P^2_k,
\]
\[
\overline{W}_a := \{XZ^{p-1} = -(X^p + aY^p)\} \subset \P^2_k.
\]
The points at infinity on these curves are not $k$-rational, since $a \notin k^p$.

We may construct an extension $U_a$ of $V_a$ by $W_a$ by using the following nonzero alternating bi-additive $2$-cocycle $h: V_a \times V_a \rightarrow W_a$:
\begin{equation}\label{hadd}
h\left( (x, y), (x', y')\right) := (xx'^p-x^px', xy'^p-x'y^p).
\end{equation}
(For generalities on the relationship between extensions of algebraic groups and $2$-cocycles, see \cite[Chap.\,II, \S 3.2]{demazure}. We will not require any of these general results about this relationship.)
We let 
\begin{equation}\label{uadef}
U_a := W_a \times V_a
\end{equation}
as $k$-schemes, with group law
\begin{equation}\label{ugplaw}
(w, v)\cdot(w', v') := (w + w' + h(v, v'), v + v').
\end{equation}
This defines a group law with identity $(0, 0)$ and inverse $(w, v)^{-1} = (-w, -v)$, and projection onto $V_a$ is a surjective group homomorphism with kernel identified with $W_a$ via the map $w \mapsto (w, 0)$. Further, if $p > 2$ then $U_a$ is non-commutative. It is wound unipotent because $V_a, W_a$ are. Note that $U_a(k) \rightarrow V_a(k)$ is surjective, since $U_a = W_a \times V_a$ as $V_a$-schemes.

When $p = 2$, the construction is somewhat more complicated. First consider general $p$ (not necessarily equal to $2$), and consider the group
\[
W^+_a := \{ X = X^p + aY^p\} \subset \Ga^2.
\]
Consider the $k$-morphism $b: V_a \rightarrow W^+_a$ defined by $b(x, y) := (x^{p+1}, xy^p)$ and the symmetric bi-additive $2$-coboundary $h^+ := -db: V_a \times V_a \rightarrow W^+_a$ defined by
\begin{equation}\label{hplus}
h^+(v, v') := b(v + v') - b(v) - b(v') = (xx'^p+x^px', xy'^p+x'y^p),
\end{equation}
where $v = (x, y)$, $v' = (x', y')$. Now choose $\zeta \in \F_{p^2} - \F_p$, and consider the bi-additive $2$-cocycle $h_{\zeta}: V_a \times V_a \rightarrow W^+_a$ defined by 
\begin{equation}\label{hzeta}
h_{\zeta}(v, v') := h^+(v, \zeta v') = h^+(\zeta^pv, v').
\end{equation}
 This is not symmetric, hence defines a non-commutative group $U^{\zeta}_a$ as follows. Let 
 \begin{equation}\label{uazeta}
 U^{\zeta}_a = W^+_a \times V_a
 \end{equation}
 as $k$-schemes, with group law given by 
\begin{equation}\label{uzetagp}
(w, v) \cdot (w', v') := (w + w' + h_{\zeta}(v, v'), v + v').
\end{equation}
The identity of $U^{\zeta}_a$ is $(0, 0)$, and inversion is given by $(w, v)^{-1} = (-w - h_{\zeta}(v, -v), -v)$. Further, projection onto $V_a$ is a surjective group homomorphism with kernel identified with $W^+_a$ via the map $w \mapsto (w, 0)$. Unfortunately, this only defines $U^{\zeta}_a$ over $\F_{p^2}(a)$. 

\begin{remark}
\label{F_4notink}
For the counterexamples that we will construct in \S \ref{sectionpathologies1}, the groups $U_a$ constructed above whenever $p > 2$ or $p = 2$ and $\F_4 \subset k$ will suffice, since we will only deal with global function fields containing $\F_{p^2}$. In \S \ref{sectionpathologies2}, we will construct examples of unipotent groups with strange behavior over {\em every} global function field, so we will require the groups constructed in a more complicated manner below when $p = 2$ but $\F_4 \not \subset k$. The reader who wishes to avoid this complication may simply restrict attention to fields of characteristic $p > 2$ and fields of characteristic $2$ containing $\F_4$. In this case, Proposition \ref{connectingmap} simplifies as explained in Remark \ref{oddassumptions}, and one can ignore the more complicated case when $\F_4 \not \subset k$ in the proof of Lemma \ref{connectingmap_a}.
\end{remark}

In order to deal with fields $k$ of characteristic $2$ such that $\F_4 \not \subset k$, we now define a Galois descent datum on $U^{\zeta}_a$ in order to descend it to an extension of $V_a$ by $W^+_a$ over the rational function field $\F_2(a) \subset k$. 

Let $p = 2$, so $W^+_a = W_a$. Let $\sigma$ denote the nontrivial automorphism of $\F_4(a)$ over $\F_2(a)$. Then $\zeta$ is a primitive cube root of unity, hence its Galois conjugate over $\F_2$ is $\zeta^{-1} = \zeta + 1$. For the Galois conjugate group $U^{\zeta + 1}_a$, note that 
\[
h_{\zeta + 1} = h_{\zeta} + h^+ = h_{\zeta} - db,
\]
so we get an $\F_4(a)$-isomorphism $[\sigma]: U^{\zeta}_a \xrightarrow{\sim} U^{\zeta + 1}_a \simeq \sigma^*(U^{\zeta}_a)$ defined by $$(w, v) \mapsto (w + b(v), v).$$ One checks that $\sigma^*([\sigma]) \circ [\sigma]: U^{\zeta}_a \rightarrow U^{\zeta}_a$ is the identity map, so $[\sigma]$ defines a descent datum on $U^{\zeta}_a$, which therefore (because of effectivity of descent for affine schemes) descends to a non-commutative extension of $V_a$ by $W_a$ over $\F_2(a)$, which we again denote by $U_a$. The corresponding group over $k$ is just the base change of this one from $\F_2(a)$ to $k$. Note that $U_a(k) \rightarrow V_a(k)$ is surjective if $\F_4 \subset k$, since then $U_a = W_a \times V_a$ as $V_a$-schemes.

\section{Pathologies with unipotent groups: Tamagawa numbers and exact sequences}
\label{sectionpathologies1}

The purpose of this section is to prove Theorems \ref{tamformulafails} and \ref{complexnotexact}. In order to do this we will make use of Gabber's groups $U_a$ defined in \S \ref{gabbergps}. In particular, we take $k := \F_q(T)$, $q := p^{2n}$ for $n \geq 1$, and $a := T(T-1)$. Denote the groups $W_a, V_a, U_a$ simply by $W, V, U$, respectively. This notation will be in force throughout the rest of this section. We will show that the conclusions of Theorems \ref{tamagawaformula} and \ref{pseudoredcomplexexact} fail for $U$. In order to do this, we begin with some calculations.

\begin{lemma}
\label{Sha=0}
$\Sha(V) = 0$.
\end{lemma}

\begin{proof}
Using the exact sequence
\[
0 \longrightarrow V \longrightarrow \Ga^2 \xlongrightarrow{f} \Ga \longrightarrow 0,
\]
where $f(x, y) := x - x^{p^2} - T(T-1)y^{p^2}$, we see that ${\rm{H}}^1(k, V) \simeq k/f(k^2)$, and similarly for $k_v$. Therefore,
\[
\Sha(V) \simeq \frac{\{ \lambda \in k \mid \lambda \in f(k_v^2) \mbox{ for all } v\}}{f(k^2)}.
\]
So suppose that $\lambda \in k$ lies in $f(k_v^2)$ for every place $v$ of $k$. We want to show that $\lambda \in f(k^2)$. For each $v$, write $\lambda = f(x_v, y_v)$ for some $x_v, y_v \in k_v$. By strong approximation, we may choose $x, y \in k$ such that $x - x_v, y - y_v \in \calO_v$ for all $v \neq \infty$ such that $\lambda \notin \calO_v$ and such that $x, y \in \calO_v$ for all other $v \neq \infty$. Then $\lambda - f(x, y) = f(x_v - x, y_v - y) \in \calO_v$ for all $v \neq \infty$, so by replacing $\lambda$ by $\lambda - f(x, y)$, we may assume that $\lambda \in \calO_v$ for all $v \neq \infty$. 

Renaming, we have $\lambda = f(x_{\infty}, y_{\infty})$ for some $x_{\infty}, y_{\infty} \in k_{\infty}$. I 
claim that for any $z \in k_{\infty}$, there exists $\alpha \in k$ such that $\alpha \in \calO_v$ for all $v 
\neq \infty$ and ${\rm{ord}}_{\infty}(\alpha - z) > 0$. Indeed, writing $z = \sum_{n \geq -N} c_nT^{-n}$ 
for some $c_n \in \F_q$, we may take $\alpha := \sum_{n \leq 0} c_nT^{-n}$. Applying this to 
$x_{\infty}, y_{\infty}$, we see that there exist $x, y \in k$ such that $x, y \in \calO_v$ for all $v \neq 
\infty$, while ${\rm{ord}}_{\infty}(x - x_{\infty})$, ${\rm{ord}}_{\infty}(y - y_{\infty}) > 0$. Then $\lambda - 
f(x, y) \in \calO_v$ for $v \neq \infty$, while $\lambda - f(x, y) = (x_{\infty} - x) - (x_{\infty} - x)^{p^2} - 
T(T-1)(y_{\infty} - y)^{p^2} \in \mathfrak{m}_{\infty}$, the maximal ideal of $\calO_{\infty}$. It follows that 
$\lambda - f(x, y) = 0$.
\end{proof}

\begin{lemma}
\label{W(k)}
If $p > 3$, then
\[
W(k) = \{ (\lambda, 0) \in k \times k \mid \lambda + \lambda^p = 0\}.
\]
If $p = 3$, then
\[
W(k) = \left\{ \left(\lambda + \frac{\mu}{T+1}, \frac{\mu}{T+1}\right) \in k \times k \middle| \lambda + \lambda^3 = \mu + \mu^3 = 0\right\}.
\]
In particular, 
\[
\#W(k) = \begin{cases}
9, & p = 3, \\
p, & p > 3.
\end{cases}
\]
\end{lemma}

\begin{remark}
If $p = 2$, then $W$ is a smooth affine plane conic, hence rational, so $W(k)$ is infinite.
\end{remark}

\begin{proof}
The last assertion follows from the first two and the fact that $x + x^p = 0$ has no repeated roots and all of its roots lie in $\F_{p^2} \subset \F_q$, since they satisfy $x^{p^2} = (x^p)^p = (-x)^p = -x^p = x$.

In order to prove that the points listed are all of the $k$-points of $W$, we first claim that if $(x, y) \in W(k)$ and $p > 3$, then ${\rm{ord}}_v(x) \geq 0$ for all places $v$ of $k$, hence $x \in \F_q$. Indeed, suppose to the contrary that some ${\rm{ord}}_v(x) < 0$ (so $x \neq 0$). Then ${\rm{ord}}_v(x + x^p) = {\rm{ord}}_v(x^p) = p\cdot {\rm{ord}}_v(x) < 0$, so $x + x^p \neq 0$. Using the equation
\begin{equation}
\label{deqn2}
x + x^p = -T(T-1)y^p
\end{equation}
(forcing $y \neq 0$), we see that $p\cdot {\rm{ord}}_v(x) = {\rm{ord}}_v(T(T-1)) + p\cdot {\rm{ord}}_v(y)$, so $p \mid {\rm{ord}}_v(T(T-1))$. Since $p > 2$, this implies that ${\rm{ord}}_v(T(T-1)) = 0$, and in particular, $v \neq 0, 1, \infty$. We deduce that ${\rm{ord}}_v(x) = {\rm{ord}}_v(y)$.

Now taking differentials of (\ref{deqn2}) for the field extension $k = \F_q(T)$ over $\F_q$ yields $dx = y^pd(T(T-1)) = (2T-1)y^pdT \neq 0$. In particular, ${\rm{ord}}_v(dx) = p\cdot {\rm{ord}}_v(x) + {\rm{ord}}_v((2T-1)dT)$. The last quantity on the right side of this equation is at most $1$, hence if $p > 3$, then (since ord$_v(x) \leq -1$)
\begin{equation}
\label{ineq1}
{\rm{ord}}_v(dx) = p\cdot {\rm{ord}}_v(x) + {\rm{ord}}_v((2T-1)dT) < 3\cdot {\rm{ord}}_v(x) + 1 \leq {\rm{ord}}_v(x) - 1
\end{equation}
and this is a contradiction, since we always have ${\rm{ord}}_v(dx) \geq {\rm{ord}}_v(x) - 1$. Thus, if $p > 3$, then $x \in \F_q$. If $y \neq 0$, then it would follow that $T(T-1) \in k^p$, a contradiction. Therefore, $y = 0$ and $x + x^p = 0$.

If $p = 3$, then the strict inequality between the outer terms in (\ref{ineq1}) still holds unless ${\rm{ord}}_v(x) = -1$ and ${\rm{ord}}_v((2T-1)dT) = 1$, i.e., $v = -1$. Thus, if $p = 3$, then $x \in \calO_v$ for all $v \neq -1$ and ${\rm{ord}}_{-1}(x) \geq -1$. It follows that $x = \lambda + \mu/(T+1)$ for some $\lambda, \mu \in \F_q$. Plugging this into (\ref{deqn2}), we find that
\begin{equation}
\label{deqn5}
\lambda + \lambda^3 + \frac{\mu}{T+1} + \frac{\mu^3}{(T+1)^3} = -T(T-1)y^3.
\end{equation}
Now we claim that $y \in \calO_v$ for $v = 0, 1$. Indeed, otherwise the left side of (\ref{deqn5}) lies in $\calO_v$ while the right side does not. So we may reduce (\ref{deqn5}) modulo $\mathfrak{m}_0$ and $\mathfrak{m}_1$, and doing so yields two equations in $\lambda + \lambda^3$ and $\mu + \mu^3$ which we solve to find that both equal $0$. This shows that $x$ is of the type asserted in the lemma. One then simply solves for $y$.
\end{proof}

\begin{lemma}
\label{Extnonzero}
$\Ext^1(W, \Gm) \neq 0$.
\end{lemma}

\begin{remark}
\label{Extnonzeroremark}
One can show, by an argument similar to the one used to prove Lemma \ref{Sha=0}, that $\Sha(W) = 0$. In conjunction with the argument in the proof below and the fact that $\Ext^1(W, \Gm)$ is $p$-torsion (because $W$ is), this then shows that in fact
\[
\Ext^1(W, \Gm) = \begin{cases}
\Z/2\Z, & p = 2, \\
(\Z/p\Z)^2, & p > 2.
\end{cases}
\]
We will never use this.
\end{remark}

\begin{proof}
Making the change of variables $X \mapsto -X/T(T-1)$, we see that 
\begin{equation}
\label{eqndef1}
W \simeq \{ Y^p = X + (T(T-1))^{p-1}X^p \}.
\end{equation}
First suppose that $p = 2$. By \cite[Prop.\,5.12]{rospic}, $\Ext^1(W, \Gm) = \Pic(W)$. By \cite[Prop.\,5.4]{rospic}, therefore, in order to show that $\Ext^1(W, \Gm) \neq 0$, it is enough (in fact, equivalent) to show that if $C$ is the regular compactification of the smooth affine curve $W$, then the unique point $Q$ of $C - W$ is not $k$-rational. We claim that the projectivization 
\begin{equation}
\label{Extnonzerolemmacurve}
\{Y^2 = XZ + (T(T-1))X^2\} \subset \P^2_k
\end{equation}
of the equation (\ref{eqndef1}) for $W$ is a regular curve, hence is the regular compactification $C$ sought. The point $Q$ at $\infty$ is the one defined in the affine patch $X \neq 0$ by the equation $Y^2 = T(T-1)$, which is of course not a rational point. This would complete the proof when $p = 2$ (and would in fact show that $\Ext^1(W, \Gm) \simeq \Z/2\Z$ in this case by \cite[Props.\,5.2, 5.4]{rospic}, since the regular compactification $C$ of $W$ has genus $0$).

It remains in the case $p = 2$ to check that the projective curve (\ref{Extnonzerolemmacurve}) is regular. This is immediate away from the locus at infinity (defined by $Z = 0$), since on that locus the curve is just $W$, which is geometrically isomorphic to $\Ga$, hence smooth. It remains to check regularity at the point at infinity. For this, we rewrite the equation in the affine chart on which $X$ does not vanish to yield the affine curve
\[
S: y^2 = z + T(T-1).
\]
We need to check that $S$ is regular at the unique point at which $z = 0$. At this point, the function $z$ vanishes. The closed subscheme $S(z)$ of $S$ defined by the ideal $z$ generates is the spectrum of $k[y]/(y^2 - T(T-1))$, which is a field, because $T(T-1) \in k$ is not a perfect square. It follows that the maximal ideal at the point at infinity on $S$ is generated by the function $z$. Since this maximal ideal is generated by a single element, it follows that $S$ is regular at the point at infinity (i.e., the zero locus of $z$), as desired. This completes the proof for $p = 2$.

Now suppose that $p > 2$. We will show that $\Ext^1(W, \Gm) \neq 0$ by computing $\tau(W)$, and in particular we will show that 
\[
\tau(W) = p^2 > 1.
\]
This is sufficient by Theorem \ref{tamagawaformula}. In order to do this, we apply \cite[Chap.\,VI, \S 7.5, Prop.]{oesterle}, which says that
\begin{equation}
\label{tameqn1}
\tau(W) = \frac{q^{1 - g +N}p^l}{\# W(k)},
\end{equation}
where $g$ is the genus of the curve $X = \P^1_{\F_q}$ of which $k$ is the function field ($g = 0$ in this case), and $N$ and $l$ are defined as follows. Let $b = (T(T-1))^{p-1}$ for notational simplicity. We have
\[
N = \sum_v \left[\frac{{\rm{ord}}_v(db)}{p(p-1)}\right] [k(v): \F_q],
\]
where the sum is over all places $v$ of $k$, the brackets denote the maximum integer function, and $k(v)$ is the residue field of $X$ at $v$. The integer $l$ is defined to be the number of places $v$ of $k$ such that the following holds: the quantity ${\rm{ord}}_v(db) + 1$ is a multiple $m(p-1)$ of $p-1$, and for some (equivalently, any) uniformizer $\pi$ at $v$, the image of the element
\begin{equation}
\label{diffeqn}
\frac{\pi^{1-m(p-1)}}{m} \frac{db}{d\pi}
\end{equation}
in $k(v)$ is a $(p-1)$st power. (The integer $m$ is coprime to $p$, since one cannot have ord$_v(db) \equiv -1 \pmod p$ in characteristic $p$ by a local calculation with power series: if $b = \sum_{n \geq N} c_n\pi^n$ with $c_n \in k(v) \subset k_v$, then $db = \sum_{n \geq N} nc_n\pi^{n-1}d\pi$, so ord$_v(db) = {\rm{min}} \{n-1 \mid nc_n \neq 0\}$, and the minimal such $n$ is obviously nonzero modulo $p$.)

First, we compute that 
\[
db = -(2T-1)(T(T-1))^{p-2}dT, 
\]
so
\[
{\rm{ord}}_v(db) = \begin{cases}
1, & v = 1/2, \\
p-2, & v = 0,1, \\
1 - 2p, & v = \infty, \\
0, & \mbox{otherwise}.
\end{cases}
\]
Thus, $N = -1$. We still need to compute $l$. If $p > 3$, then the only places $v$ for which ${\rm{ord}}_v(db) + 1$ is a multiple of $p-1$ are $v = 0, 1, \infty$. If $p = 3$, then this also holds for $v = 1/2 = -1$. We will check below that in all of these cases, the quantity (\ref{diffeqn}) is a $(p-1)$st power in $k(v)$, so that $l = 3$ for $p > 3$ and $l = 4$ when $p = 3$. Applying (\ref{tameqn1}), therefore, together with Lemma \ref{W(k)}, we see that $\tau(W) = p^2$ for $p > 2$, as desired.

It remains to check that (\ref{diffeqn}) is a $(p-1)$st power in $k(v)$ in all of the cases listed above. In fact, a straightforward computation using the uniformizers $\pi = T, T-1, T^{-1}$, and $2T - 1$ at the places $v = 0, 1, \infty$, and $1/2$ (when $p=3$), respectively, shows that the image of (\ref{diffeqn})  in $k(v)$ equals $-1$ in these cases. Further, the residue field in all of these cases is $\F_q$. The element $-1$ is a $(p-1)$st power in $\F_q$ precisely when $(-1)^{(q-1)/(p-1)} = 1$. We have $(-1)^{(q-1)/(p-1)} = (-1)^{1 + p + \dots + p^{2n-1}} = 1$, so $-1$ is indeed a $(p-1)$st power in $\F_q$.
\end{proof}

For a smooth connected group scheme $G$ over $k$, let $\mathscr{D}G$ denote the derived group of $G$, and $G^{\rm{ab}} := G/\mathscr{D}G$ the abelianization of $G$.

\begin{lemma}
\label{Ext^1=Ext^1ab}
Let $G$ be a smooth connected unipotent group over an arbitrary field $k$. Then the map ${\rm{H}}^1(k, \widehat{G}) \rightarrow \Ext^1(G, \Gm)$ appearing in the exact sequence ($\ref{H^1(G^)=kerExt}$) is an isomorphism. Further, the pullback map $\Ext^1(G^{{\rm{ab}}}, \Gm) \rightarrow \Ext^1(G, \Gm)$ is an isomorphism, and the map $\Ext^1(G, \Gm) \rightarrow \Ext^1(\mathscr{D}G, \Gm)$ is $0$.
\end{lemma}

\begin{proof}
The first assertion implies the others, since $\widehat{G} = \widehat{G^{{\rm{ab}}}}$. To prove the first assertion, it suffices to show that $\Ext^1_{\overline{k}}(G, \Gm) = 0$. In fact, $\Pic(G_{\overline{k}}) = 0$ because $G_{\overline{k}}$ is split unipotent, hence isomorphic as a $\overline{k}$-scheme to some affine $n$-space, so we are done.
\end{proof}

We will repeatedly use the following formula due to Oesterl\'e for the behavior of Tamagawa numbers in exact sequences.

\begin{lemma}
\label{tamsinsequences}
Suppose given an exact sequence of connected linear algebraic groups over a global field $k$:
\[
1 \longrightarrow G' \xlongrightarrow{j} G \xlongrightarrow{\pi} G'' \longrightarrow 1
\]
such that $\pi(G(\A))$ is normal in $G''(\A)$. Then
\[
\tau(G) \cdot \left| \frac{G''(\A)}{\pi(G(\A))G''(k)} \right| = \tau(G')\tau(G'')|\ker(\Sha^1(j))|\cdot |\coker(\widehat{j})|^{-1}.
\]
\end{lemma}

\begin{proof}
This is Lemma \ref{tamsinsequences}, except there it is stated under the assumption that $\tau(G')$, $\tau(G'')$, and $\ker(\Sha^1(j))$ are all finite. Thanks to \cite[Thm.\,1.3.3(i), Thm.\,1.3.6]{conrad}, these assumptions always hold. (Remark: we should warn the reader that \cite[Thm.\,1.3.3(i)]{conrad} is incorrect as stated. The claimed finiteness only holds in the generality claimed there for $S = \emptyset$, which is fortunately the only case we need; see \cite[Cor.\,5.19, Rmk.\,5.17]{rospic}.)
\end{proof}

We may now show that the conclusions of Theorems \ref{tamagawaformula} and \ref{pseudoredcomplexexact} fail for the group $U$.

\begin{theorem}
\label{tamfails}
For $k = \F_{p^{2n}}(T)$ and $a=T(T-1) \in k$, let $V_a$ and $W_a$ be as in $(\ref{defofVW})$.
Define the $2$-dimensional wound non-commutative $k$-group extension $U$ of $V_a$ by $W_a$ as follows:
\begin{itemize}
\item If $p>2$, then define $U=U_a$ as in $(\ref{uadef})$ with group law
$(\ref{ugplaw})$ resting on $h$ as in $(\ref{hadd})$. 
\item If $p=2$, then upon choosing a primitive cube root of unity $\zeta \in \F_4 \subset
\F_{p^{2n}}$, define $U = U_a^{\zeta}$ as in $(\ref{uazeta})$ with group law
$(\ref{uzetagp})$ resting on $h_{\zeta}$ as in $(\ref{hzeta})$ and $(\ref{hplus})$.
\end{itemize}

Then Theorem {\rm{\ref{tamagawaformula}}} fails for $U$.  That is, 
$\tau(U) \neq \# \Ext^1(U, \Gm)/\# \Sha(U)$.
\end{theorem}

\begin{proof}
We have the exact sequence
\[
1 \longrightarrow W \longrightarrow U \longrightarrow V \longrightarrow 1.
\]
The maps $U(k) \rightarrow V(k)$ and $U(\A) \rightarrow V(\A)$ are surjective. (When $p=2$, this uses the fact that $\F_4 \subset k$.) We claim that the map $\Sha(W) \rightarrow \Sha(U)$ is a bijection. We first show that even the map ${\rm{H}}^1(k, W) \rightarrow {\rm{H}}^1(k, U)$ is injective. Since $W \subset U$ is central, two elements of ${\rm{H}}^1(k, W)$ have the same image in ${\rm{H}}^1(k, U)$ if and only if they differ by an element of $\delta(V(k))$, where $\delta: V(k) \rightarrow {\rm{H}}^1(k, W)$ is the connecting map \cite[Chap.\,I, \S 5.5, Prop.\,39(ii) and \S 5.6, Cor.\,2]{serre}. But this connecting map is trivial, since $U(k) \rightarrow V(k)$ is surjective. Similarly, the map ${\rm{H}}^1(k_v, W) \rightarrow {\rm{H}}^1(k_v, U)$ is injective for all $v$. For surjectivity, we note that any element of $\Sha(U)$ maps to $\Sha(V) = 0$ (Lemma \ref{Sha=0}), hence lifts to an element of ${\rm{H}}^1(k, W)$, which must lie in $\Sha(W)$ by the injectivity of the maps ${\rm{H}}^1(k_v, W) \rightarrow {\rm{H}}^1(k_v, U)$.

Since $W$ is unipotent, $\widehat{W}(k) = 0$, so by Lemma \ref{tamsinsequences},
\[
\tau(U) = \tau(W) \tau(V).
\]
By Theorem \ref{tamagawaformula} applied to the commutative groups $V$ and $W$, and Lemma \ref{Sha=0}, we therefore obtain
\[
\tau(U) = \frac{\# \Ext^1(W, \Gm) \cdot \# \Ext^1(V, \Gm)}{\# \Sha(W)}.
\]
Now we claim that $W = \mathscr{D}U$. This may be seen directly, but it also follows from dimension considerations as follows. Since $V$ is commutative, $\mathscr{D}U \subset W$. Since $U$ is non-commutative, $\mathscr{D}U$ is a nontrivial smooth connected $k$-group, hence, since $W$ is $1$-dimensional, we must have $\mathscr{D}U = W$. Thus, $V = U^{\rm{ab}}$. By Lemma \ref{Ext^1=Ext^1ab} and the fact that $\Sha(W) \xrightarrow{\sim} \Sha(U)$ proved above, we therefore obtain
\[
\tau(U) = \frac{\# \Ext^1(W, \Gm) \cdot \# \Ext^1(U, \Gm)}{\# \Sha(U)}.
\]
Lemma \ref{Extnonzero} now shows that $\tau(U) \neq \# \Ext^1(U, \Gm)/\# \Sha(U)$.
\end{proof}

\begin{theorem}
\label{complexnotexactU}
Theorem {\rm{\ref{pseudoredcomplexexact}}} fails for the wound $2$-dimensional non-commutative group $U$ over the global field $k = \F_{p^{2n}}(T)$. That is, the complex
\[
{\rm{H}}^1(k, U) \longrightarrow {\rm{H}}^1(\A, U) \longrightarrow \Ext^1(U, \Gm)^*
\]
of pointed sets is not exact.
\end{theorem}

\begin{proof}
Once again, we have the exact sequence
\[
1 \longrightarrow W \longrightarrow U \longrightarrow V \longrightarrow 1.
\]
For any $\alpha \in {\rm{H}}^1(\A, W)$, the image of $\alpha$ in ${\rm{H}}^1(\A, U)$ maps to $0 \in \Ext^1(U, \Gm)^*$. Indeed, this follows from the commutative diagram
\[
\begin{tikzcd}
{\rm{H}}^1(\A, W) \arrow{r} \arrow{d} & {\rm{H}}^1(\A, U) \arrow{d} \\
\Ext^1(W, \Gm)^* \arrow{r} & \Ext^1(U, \Gm)^*
\end{tikzcd}
\]
in which the bottom map is $0$ by Lemma \ref{Ext^1=Ext^1ab} since $W = \mathscr{D}U$, as we saw in the proof of Theorem \ref{tamfails}. It therefore suffices to construct an element $\alpha \in {\rm{H}}^1(\A, W)$ whose image in ${\rm{H}}^1(\A, U)$ does not lift to a class in ${\rm{H}}^1(k, U)$.

We claim that an element $\alpha \in {\rm{H}}^1(\A, W)$ has image in ${\rm{H}}^1(\A, U)$ that lifts to a global class in ${\rm{H}}^1(k, U)$ if and only if $\alpha$ itself lifts to ${\rm{H}}^1(k, W)$. Clearly, if $\alpha$ lifts to ${\rm{H}}^1(k, W)$, then its image in ${\rm{H}}^1(\A, U)$ lifts to a global class. Conversely, let $j: {\rm{H}}^1(\A, W) \rightarrow {\rm{H}}^1(\A, U)$ denote the map induced by the inclusion $W \hookrightarrow U$, and suppose that $j(\alpha)$ lifts to $u \in {\rm{H}}^1(k, U)$. Then the image of $u$ in ${\rm{H}}^1(k, V)$ lies in $\Sha(V)$, which vanishes by Lemma \ref{Sha=0}. Thus, $u$ lifts to some class $w \in {\rm{H}}^1(k, W)$. Let $w_{\A}$ denote the image of $w$ in ${\rm{H}}^1(\A, W)$. Then $j(\alpha) = j(w_{\A})$. But, as we discussed in the proof of Theorem \ref{tamfails}, the map $j$ is injective due to the surjectivity of the map $U(\A) \rightarrow V(\A)$. Therefore, $\alpha = w_{\A}$. That is, $\alpha$ lifts to the class $w \in {\rm{H}}^1(k, W)$. This proves the claim.

It therefore only remains to show that $\Che^1(W) := \coker({\rm{H}}^1(k, W) \rightarrow {\rm{H}}^1(\A, W))$ is nonzero. But ${\rm{H}}^2(k, W) = 0$ \cite[Prop.\,2.5.4(i)]{rostateduality}, so by global Tate duality for affine schemes \cite[Thm.\,1.2.8]{rostateduality}, we have an isomorphism $\Che^1(W) \simeq {\rm{H}}^1(k, \widehat{W})^*$. We also have an isomorphism ${\rm{H}}^1(k, \widehat{W}) \simeq \Ext^1(W, \Gm)$ \cite[Cor.\,2.3.4]{rostateduality}, so the desired nonvanishing follows from Lemma \ref{Extnonzero}.
\end{proof}

\section{Relation between Tate-Shafarevich sets and Tamagawa numbers under inner twisting}
\label{sectionpathologies2}

The main result of this section is the following proposition, which shows that under certain hypotheses, the Tamagawa numbers and the (sizes of the) Tate-Shafarevich sets of a connected linear algebraic group $G$ are essentially inversely proportional as one varies over the inner forms of $G$. This will allow us to prove Theorem \ref{innerinvariancefailsforU} by showing that for suitable groups, $\tau$ can become arbitrarily small under inner twisting, and therefore automatically $\Sha$ becomes arbitrarily large.

\begin{proposition}
\label{tautimesSha}
Suppose that we have a central extension
\[
1 \longrightarrow G' \longrightarrow G \longrightarrow G'' \longrightarrow 1
\]
of connected linear algebraic groups over a global function field $k$ such that $G''$ is commutative. Suppose that either $\Sha(G'') = 0$ or that $G''(k)$ is finite. Then there are constants $c, d > 0$ $($depending on $G$$)$ such that for all inner forms $\widetilde{G}$ of $G$,
\[
c < \tau(\widetilde{G}) \cdot \# \Sha(\widetilde{G}) < d.
\]
\end{proposition}

\begin{proof}
Given functions $F, H$ from the set $Z^1(k, G/Z_G)$ of cocycles valued in $G/Z_G$ to the positive reals, let us write $F \approx H$ is there exist constants $c, d > 0$ such that $c\cdot H(\beta) < F(\beta) < d\cdot H(\beta)$ for all $\beta \in Z^1(k, G/Z_G)$. We have the exact sequence
\[
1 \longrightarrow G' \longrightarrow G \longrightarrow G'' \longrightarrow 1.
\]
Twisting by an element $\beta \in Z^1(k, G/Z_G)$, we obtain the sequence
\[
1 \longrightarrow G' \xlongrightarrow{j_{\beta}} G_{\beta} \xlongrightarrow{\pi_{\beta}} G'' \longrightarrow 1,
\]
where $G', G''$ are left unchanged, because $G'$ is central in $G$ and the quotient $G''$ is commutative. By Lemma \ref{tamsinsequences}, we have
\begin{equation}
\label{taueqn10}
\tau(G_{\beta}) \cdot \# \left( \frac{G''(\A)}{\pi_{\beta}(G_{\beta}(\A))G''(k)} \right) = \tau(G')\tau(G'')\cdot \# \ker(\Sha(j_{\beta})) (\# \coker(\widehat{j_{\beta}}))^{-1}, 
\end{equation}
where $\widehat{j_{\beta}}: \widehat{G_{\beta}}(k) \rightarrow \widehat{G'}(k)$ is the induced map on character groups. We claim that the right side above is $\approx 1$. Indeed, we first note that $\ker(\Sha(j_{\beta})) \subset \Sha(G')$, hence its cardinality is bounded. We also need to check that $\# \coker(\widehat{j_{\beta}}) \approx 1$. This holds because $\widehat{G_{\beta}}(k) = \widehat{G_{\beta}^{{\rm{ab}}}}(k)$, and inner twisting has no effect on the abelianization, so that in fact $\# \coker(\widehat{j})$ is invariant under inner twisting.

Since the right side of (\ref{taueqn10}) is $\approx 1$, we need to show that
\begin{equation}
\label{eqnapprox1}
\# \left( \frac{G''(\A)}{\pi_{\beta}(G_{\beta}(\A))G''(k)} \right) \stackrel{?}{\approx} \# \Sha(G_{\beta}).
\end{equation}
Let us endow $\ker(\Sha(\pi_{\beta}))$ with the structure of abelian group as follows. Any element of $\ker(\Sha(\pi_{\beta}))$ lifts to ${\rm{H}}^1(k, G')$. Since $G' \subset G_{\beta}$ is central, the abelian group ${\rm{H}}^1(k, G')$ acts on the set ${\rm{H}}^1(k, G)$; denoting this action by $*$, the map ${\rm{H}}^1(k, G') \rightarrow {\rm{H}}^1(k, G)$ is given by $\alpha \mapsto \alpha * 1$ \cite[Chap.\,I, \S 5.7]{serre}, and similar statements hold for adelic cohomology thanks to \cite[Prop.\,1.5]{rospred}. Therefore, we see that the map
\[
\frac{\{ x \in {\rm{H}}^1(k, G') \mid x_{\A} * 1 = 1 \in {\rm{H}}^1(\A, G)\}}{\{ x \in {\rm{H}}^1(k, G') \mid x * 1 = 1 \}} \rightarrow \ker(\Sha(\pi_{\beta}))
\]
is a bijection from the abelian group on the left to the set on the right. Thus, by transfer of structure, we obtain an abelian group structure on $\ker(\Sha(\pi_{\beta}))$.

We will now construct an exact sequence of finite abelian groups
\begin{equation}
\label{exactseq5}
\Sha(G') \longrightarrow \ker(\Sha(\pi_{\beta})) \xlongrightarrow{\psi} \frac{G''(\A)}{\pi_{\beta}(G_{\beta}(\A))G''(k)} \xlongrightarrow{\phi} \Che^1(G'),
\end{equation}
where we recall that $\Che^1(G') := \coker({\rm{H}}^1(k, G') \rightarrow {\rm{H}}^1(\A, G'))$ is finite by \cite[Chap.\,IV, \S 2.6, Prop.\,(b)]{oesterle}; note that this definition of $\Che^1(G')$ agrees with the definition $\Che^1(G') := \coker({\rm{H}}^1(k, G') \rightarrow \oplus_v {\rm{H}}^1(k_v, G'))$ given in \cite{oesterle} by \cite[Prop.\,1.5]{rospred}. The map $\Sha(G') \rightarrow \ker(\Sha(\pi_{\beta}))$ is the one induced by the map $G' \rightarrow G$. The map $\phi$ is the one induced by the connecting map $G''(\A) \rightarrow {\rm{H}}^1(\A, G')$, which is a homomorphism because $G' \subset G$ is central, by \cite[Ch.\,I, \S 5.6, Cor.\,2]{serre} and \cite[Prop.\,1.5]{rospred}. To define $\psi$, consider the following exact diagram of pointed sets
\[
\begin{tikzcd}
& G''(k) \arrow{r}{\delta_{\beta}} \arrow{d} & {\rm{H}}^1(k, G') \arrow{r}{j} \arrow{d} & {\rm{H}}^1(k, G_{\beta}) \arrow{r}{{\rm{H}}^1(\pi_{\beta})} \arrow{d} & {\rm{H}}^1(k, G'') \\
G_{\beta}(\A) \arrow{r}{\pi_{\beta}} & G''(\A) \arrow{r}{(\delta_{\beta})_{\A}} & {\rm{H}}^1(\A, G') \arrow{r} & {\rm{H}}^1(\A, G_{\beta}) &
\end{tikzcd}
\]
Given $\alpha \in \ker(\Sha(\pi_{\beta}))$, lift $\alpha$ to an element $w \in {\rm{H}}^1(k, G')$. Then the image $w_{\A}$ of $w$ in ${\rm{H}}^1(\A, G')$ maps to $0 \in {\rm{H}}^1(\A, G_{\beta})$, hence lifts to some element $v \in G''(\A)$. We then define $\psi(\alpha)$ to be the class of $v$ in $G''(\A)/\pi_{\beta}(G_{\beta}(\A))G''(k)$. One easily checks that this is well-defined, independent of the choice of lifts $w, v$. This uses the fact that two elements of ${\rm{H}}^1(k, G')$ have the same image in ${\rm{H}}^1(k, G_{\beta})$ if and only if they differ by an element of $\delta_{\beta}(G''(k))$ (since $G' \subset G_{\beta}$ is central \cite[Chap.\,I, \S 5.6, Cor.\,2]{serre}), and two elements of $G''(\A)$ have the same image under $(\delta_{\beta})_{\A}$ if and only if they differ by an element of $\pi_{\beta}(G_{\beta}(\A))$ due to \cite[Chap.\,I, \S 5.4, Cor.\,1]{serre} and \cite[Prop.\,1.5]{rospred}. 

We need to check that the maps in (\ref{exactseq5}) are group homomorphisms. The map $\Sha(G') \rightarrow \ker(\Sha(\pi_{\beta}))$ is by the definition of the group structure on $\ker(\Sha(\pi_{\beta}))$. To see that the maps $\psi, \phi$ are group homomorphisms, again using the definition of the group structure on $\ker(\Sha(\pi_{\beta}))$ in the case of $\psi$, it suffices to note that the connecting map $(\delta_{\beta})_{\A}$ is a group homomorphism, because $G' \subset G_{\beta}$ is central (\cite[Chap.\,I, \S 5.6, Cor.\,2]{serre} and \cite[Prop.\,1.5]{rospred}).

Now we check exactness of the sequence (\ref{exactseq5}). First, if $\alpha \in \ker(\Sha(\pi_{\beta}))$ lifts to $\Sha(G')$, then it is clear that $\psi(\alpha) = 0$, as we may then take $w \in \Sha(G')$, and $v = 0$ in the definition of $\psi(\alpha)$ above. Conversely, suppose that $\alpha \in \ker(\Sha(\pi_{\beta}))$ satisfies $\psi(\alpha) = 0$. In terms of the definition of $\psi$ given above, this means that the element $v \in G''(\A)$ lifting $w_{\A}$ lies in $\pi_{\beta}(G_{\beta}(\A))G''(k)$. Modifying $v$ by an element of $\pi_{\beta}(G_{\beta}(\A))$, therefore, as we may, we may assume that $v$ lifts to some element $v' \in G''(k)$. But then modifying the element $w \in {\rm{H}}^1(k, G')$ by $\delta_{\beta}(v')$ - again, as we may - we may assume that $w \in \Sha(G')$. That is, $\alpha$ lifts to $\Sha(G')$. This proves exactness at $\ker(\Sha(\pi_{\beta}))$.

Next we check exactness at $G''(\A)/\pi_{\beta}(G_{\beta}(\A))G''(k)$. First, it is clear from the definition that $\phi \circ \psi = 0$, since in the above notation, $\phi \circ \psi(\alpha)$ is the class of  $(\delta_{\beta})_{\A}(v) = w_{\A}$ in $\Che^1(G')$, which is $0$. Conversely, suppose that we have a class in $G''(\A)/\pi_{\beta}(G_{\beta}(\A))G''(k)$ represented by $v \in G''(\A)$ such that $\phi(v) = 0 \in \Che^1(G')$; that is, $(\delta_{\beta})_{\A}(v) = w_{\A}$ for some $w \in {\rm{H}}^1(k, G')$. Then by definition, the class of $v$ is $\psi(j(w))$ with $j(w) \in \ker(\Sha(\pi_{\beta}))$. So (\ref{exactseq5}) is exact.

The exactness of (\ref{exactseq5}) implies that $\# \ker(\psi)$ and $\# \coker(\psi)$ are both $\approx 1$, hence
\[
\# \left( \frac{G''(\A)}{\pi_{\beta}(G_{\beta}(\A))G''(k)}\right) \approx \# \ker(\Sha(\pi_{\beta})).
\]
Therefore, in order to prove (\ref{eqnapprox1}), and hence the lemma, it is the same to show that
\begin{equation}
\label{shaapproxkereqn}
\# \Sha(G_{\beta}) \approx \# \ker(\Sha(\pi_{\beta})).
\end{equation}
When $\Sha(G'') = 0$, we have $\Sha(G_{\beta}) = \ker(\Sha(\pi_{\beta}))$, so (\ref{shaapproxkereqn}) is immediate. So we now assume that $G''(k)$ is finite and prove (\ref{shaapproxkereqn}) in this case.

In order to do this, it suffices to show that the nonempty fibers of the map $\Sha(\pi_{\beta}): \Sha(G_{\beta}) \rightarrow \Sha(G'')$ all have size $\approx \# \ker(\Sha(\pi_{\beta}))$, i.e., all nonempty fibers have about the same size. (More precisely, they all have size bounded above and below by positive constants times the size of the fiber above the trivial element, where the constants depend only on $G$, not on $\beta$.)

For this, we note that for any $x \in \Sha(G_{\beta})$, the elements of $\Sha(G_{\beta})$ lying in the same fiber as $x$ are those of the form $\alpha * x$ where $\alpha \in {\rm{H}}^1(k, G')$ satisfies $\alpha_{\A} * 1 = 1 \in {\rm{H}}^1(\A, G_{\beta})$, since this is the same as $\alpha_{\A} * x_{\A} = 1$, because $x_{\A} = 1$ (as $x \in \Sha(G_{\beta})$). Two such elements $\alpha, \alpha'$ satisfy $\alpha * x = \alpha' * x$ if and only if $(\alpha - \alpha')* x = x$. Twisting the exact sequence
\[
1 \longrightarrow G' \longrightarrow G_{\beta} \longrightarrow G'' \longrightarrow 1
\]
by a cocycle $x'$ representing the cohomology class $x$ to obtain a new sequence
\begin{equation}
\label{twistseqx}
1 \longrightarrow G' \longrightarrow G_{x'} \longrightarrow G'' \longrightarrow 1,
\end{equation}
this amounts to saying $\alpha - \alpha' \mapsto 1 \in {\rm{H}}^1(k, G_{x'})$. But this in turn is equivalent to 
the condition $\alpha - \alpha' \in \delta_{x'}(G''(k))$, where $\delta_{x'}$ is the connecting map associated to the sequence (\ref{twistseqx}). Thus, we obtain a bijection between the fiber of the map $\Sha(\pi_{\beta})$ which contains $x$ and
\[
\frac{\{ \alpha \in {\rm{H}}^1(k, G') \mid \alpha_{\A} * 1 = 1 \in {\rm{H}}^1(\A, G_{\beta}) \}}{\delta_{x'}(G''(k))}.
\]
Since $G''(k)$ is finite, the sizes of these sets are all 
\[
\approx \# \{ \alpha \in {\rm{H}}^1(k, G') \mid \alpha_{\A} * 1 = 1 \in {\rm{H}}^1(\A, G_{\beta}) \},
\]
which is independent of the element $x \in \Sha(G_{\beta})$. That is, the nonempty fibers are all approximately the same size. The proof of the lemma is complete.
\end{proof}

In order to apply Proposition \ref{tautimesSha} to the extension
\[
1 \longrightarrow W_a \longrightarrow U_a \longrightarrow V_a \longrightarrow 1,
\]
defined by Gabber's unipotent groups constructed in \S \ref{gabbergps}, we will need the following lemma.

\begin{lemma}
\label{V(k)finite}
The group $V_a(k)$ is finite for any global function field $k$ and any $a \in k - k^p$.
\end{lemma}

\begin{proof}
If $p >2$, then this is a special case of \cite[Chap.\,VI, \S 3.1, Thm.]{oesterle}. When $p = 2$ we need a different argument that is in the same spirit as the proof of Lemma \ref{W(k)}. Let $(x, y) \in V_a(k)$, so
\begin{equation}
\label{anothereqn1}
x + x^4 = ay^4.
\end{equation}
Suppose we are given a place $v$ of $k$ such that ${\rm{ord}}_v(x) < 0$ (so $x \notin \F_q$ and hence $y \neq 0$). Then ${\rm{ord}}_v(a) + {\rm{ord}}_v(y^4) = {\rm{ord}}_v(ay^4) = {\rm{ord}}_v(x+x^4) = {\rm{ord}}_v(x^4) = 4\cdot {\rm{ord}}_v(x)$. That is,
\begin{equation}
\label{anothereqn2}
{\rm{ord}}_v(y^4) = 4\cdot {\rm{ord}}_v(x) - {\rm{ord}}_v(a)
\end{equation}

On the other hand, taking differentials of (\ref{anothereqn1}) yields $dx = y^4da$, and $da \neq 0 \in \Omega^1_{k_v}$ since $a \notin k_v - k_v^p$ (because $k_v/k$ is a separable extension). Therefore, since ${\rm{ord}}_v(dx) \geq {\rm{ord}}_v(x) - 1$, using (\ref{anothereqn2}) we obtain
\[
{\rm{ord}}_v(da) - {\rm{ord}}_v(a) + 4\cdot {\rm{ord}}_v(x) \geq {\rm{ord}}_v(x) - 1
\]
when ord$_v(x) < 0$. This yields for all $v$ a lower bound on ${\rm{ord}}_v(x)$. Further, for all but finitely many $v$, ${\rm{ord}}_v(a) = {\rm{ord}}_v(da) = 0$, so we actually obtain ${\rm{ord}}_v(x) \geq 0$. Thus, we obtain a divisor $D$ on the curve $X$ of which $k$ is the function field such that ${\rm{div}}(x) \geq D$ for all $(x, y) \in V_a(k)$ with $x \neq 0$. Therefore, there are only finitely many possible values of $x$, hence only finitely many $k$-points of $V_a$.
\end{proof}

Lemmas \ref{tautimesSha} and \ref{V(k)finite} show that for an inner form of $U_a$, having small Tamagawa number is equivalent to having large Tate-Shafarevich set. Thus, in order to prove Theorem \ref{innerinvariancefailsforU} for the group $U_a$, it suffices to find an inner form having one of these properties, and the other follows automatically. We will force $\tau$ to be small directly, thereby also obtaining largeness of $\Sha$. In order to do this, it is essential that we be able to compute the connecting map $V_a(k) \rightarrow {\rm{H}}^1(k, W_a)$ arising from twists of the exact sequence
\[
1 \longrightarrow W_a \longrightarrow U_a \longrightarrow V_a \longrightarrow 1.
\]
Since it is no more difficult, we will do this in greater generality than for the groups $U_a$. We will carry out this computation in the next section, but we first discuss the general situation here in order to motivate the work carried out in \S \ref{sectioninnertwistpathologies}. This discussion is purely for pedagogical purposes, and will not be used in any way in \S \ref{sectioninnertwistpathologies}.

So generalizing for now, let us temporarily assume that we have a central extension of unipotent groups over a field $k$
\begin{equation}
\label{exactseq7}
1 \longrightarrow W \xlongrightarrow{t} U \xlongrightarrow{\xi} V \longrightarrow 1
\end{equation}
with $W, V$ smooth connected $1$-dimensional unipotent $k$-groups. (The assumption on dimension is not so important; we make it mainly for simplicity of exposition and because this assumption will hold for the groups $W_a, V_a$ to which we will apply this discussion.) In particular, $W$ and $V$ are $p$-torsion; indeed, $[p]: W \rightarrow W$ is not surjective, so the image is smooth connected of dimension $< 1$, i.e., it is trivial, and similarly for $V$. (In fact, one can show that any one-dimensional smooth connected unipotent group over a field is a form of $\Ga$.) For any such groups, by \cite[Prop.\,B.1.13]{cgp} there exist exact sequences
\begin{equation}
\label{exactseq8}
0 \longrightarrow W \xlongrightarrow{l} \Ga^2 \xlongrightarrow{g} \Ga \longrightarrow 0,
\end{equation}
\[
0 \longrightarrow V \longrightarrow \Ga^2 \xlongrightarrow{f} \Ga \longrightarrow 0.
\]
In this way, we obtain identifications ${\rm{H}}^1(k, W) \simeq k/g(k^2)$ and ${\rm{H}}^1(k, V) \simeq k/f(k^2)$. 

\begin{remark}
\label{twistingcocycle}
Suppose given a class $\alpha \in {\rm{H}}^1(k, U)$. Then we may twist the sequence (\ref{exactseq7}) by $\alpha$ to obtain a new sequence. Actually, strictly speaking we must twist by a cocycle representing $\alpha$. The isomorphism class of this twist (as an extension of $V$ by $W$) is independent of the choice of cocycle representing $\alpha$, up to non-canonical isomorphism. Since this isomorphism class is all that matters for our purposes, we can abuse notation and speak of ``twisting by $\alpha$''. Further, since $W \subset U$ is central, twisting by $\alpha$ is the same as twisting by the image $\overline{\alpha}$ of $\alpha$ in ${\rm{H}}^1(k, V) \simeq k/f(k^2)$; again, strictly speaking we are twisting by a cocycle representing this image, but this twist, too, is independent of the cocycle representing this image. All of these assertions follow from the fact that the action of $U$ on the extension class of $U$ as an extension of $V$ by $W$ factors as the composition $U \rightarrow V \rightarrow \underline{{\rm{Aut}}}_{(U, V, W)/k}$ (where this last symbol denotes the automorphism functor of the extension $U$ of $V$ by $W$), and the fact that fppf forms of this extension class are classified by ${\rm{H}}^1(k, \underline{{\rm{Aut}}}_{(U, V, W)/k})$.
\end{remark}

Represent the image of $\alpha$ in ${\rm{H}}^1(k, V) \simeq k/f(k^2)$ by some $\beta \in k$, so that we obtain a twisted sequence
\begin{equation}
\label{twist2}
1 \longrightarrow W \longrightarrow U_{\beta} \longrightarrow V \longrightarrow 1.
\end{equation}
We want to compute the connecting map $\delta_{\beta}: V(k) \rightarrow {\rm{H}}^1(k, W)$. 

Consider the following pushout diagram with exact rows and columns.
\begin{equation}
\label{adiagram13}
\begin{tikzcd}
1 \arrow{r} & W \arrow{r}{t} \arrow[d, hookrightarrow, "l"] & U \arrow{r}{\xi} \arrow[d, hookrightarrow, "j"] & V \arrow{r} \arrow[d, equals] & 1 \\
1 \arrow{r} & \Ga^2 \arrow{r}{i} \arrow{d}{g} & M \arrow{r} \arrow{d}{\pi} & V \arrow{r} & 1 \\
& \Ga \arrow[r, equals] & \Ga &&
\end{tikzcd}
\end{equation}
Now $M$ is in particular a $\Ga^2$-torsor over $V$, and ${\rm{H}}^1(V, \Ga^2) = 0$ because $V$ is affine, so the map $M \rightarrow V$ admits a scheme-theoretic (but not necessarily group-theoretic) section $V \rightarrow M$. By translating this section by some element of $\Ga^2(k) = k^2$, we may assume that it carries the identity of $V$ to the identity of $M$. It follows that the group structure on $M$ is given by some Hochschild $2$-cocycle $h: V \times V \rightarrow \Ga^2$. (See \cite[Chap.\,II, \S 3.2]{demazure} for details.) We will not need to worry about this generality. Let us make the simplifying assumption that our $2$-cocycle is bi-additive. We now explain what we mean by all of this.

We may identify $M = \Ga^2 \times V$ as $\Ga^2$-torsors over $V$, with $0 \in M(k)$ mapping to $((0, 0), 0) \in (\Ga^2 \times V)(k)$, so $i$ is identified with the canonical inclusion $\Ga^2 \simeq \Ga^2 \times \{ 0_V\} \hookrightarrow \Ga^2 \times V$. Suppose that there is a bi-additive map $h: V \times V \rightarrow \Ga^2$ such that the composition law on $M$ is given by
\[
(x, v)\cdot(x', v') = (x + x' + h(v, v'), v + v').
\]
Any bi-additive $h$ as above defines a group law on $\Ga^2 \times V$ in this manner with identity $((0, 0), 0)$ and inverse $(x, v)^{-1} = (-x - h(v, -v), -v)$. We {\em assume} that the group structure on $M$ arises in this manner. Then for $\beta \in k$ representing a class in $k/f(k^2) \simeq {\rm{H}}^1(k, V)$, we would like to compute the connecting map $\delta_{\beta}: V(k) \rightarrow {\rm{H}}^1(k, W)$ associated to the twisted sequence (\ref{twist2}). This is accomplished by a lemma that involves a couple of additional assumptions. We discuss this in the next section.

\section{Pathologies with unipotent groups: inner twisting}
\label{sectioninnertwistpathologies}

In this section we will prove Theorem \ref{innerinvariancefailsforU} and thereby 
show that Theorem \ref{innerinvariance} fails dramatically beyond the commutative and pseudo-reductive cases by giving examples of wound non-commutative $2$-dimensional unipotent groups having inner forms with arbitrarily small Tamagawa number and arbitrarily large Tate-Shafarevich set. We will actually show that {\em all} of Gabber's groups $U_a$ constructed in \S \ref{gabbergps} have this property:

\begin{theorem}
\label{innerinvariancefails}
Let $k$ be a global function field, $a \in k - k^p$. Choose $\epsilon, M > 0$. Then the group $U_a$ has an inner form $U'$ such that $\tau(U') < \epsilon$ and $\# \Sha(U') > M$.
\end{theorem}

The proof will occupy this entire section. The reader who is willing to ignore fields $k$ of characteristic $2$ such that $\F_4 \not \subset k$ to avoid complications that arise in the proof of Theorem \ref{innerinvariancefails} over such fields should see Remark \ref{F_4notink}.

\begin{remark}[Questions]
Suppose that $U$ is a wound non-commutative unipotent group over a global function field $k$. Does Theorem \ref{innerinvariancefails} hold for $U$? Or does a weaker version at least hold, in which one may find inner forms of $U$ with arbitrarily small Tamagawa number, and inner forms with arbitrarily large $\Sha$, but possibly not at the same time? Is it true that $\tau(U') \cdot \# \Sha(U')$ is bounded both above and below as $U'$ varies over all inner forms $U'$ of $U$? (See Proposition \ref{tautimesSha}.)
\end{remark}

The key point is to compute connecting maps of sequences obtained by inner twisting from the exact sequences
\[
0 \longrightarrow W_a \longrightarrow U_a \longrightarrow V_a \longrightarrow 0
\]
for the groups $U_a$ constructed by Gabber. We will do this in greater generality than we require and so we first begin with a more general setup.

\underline{\textbf{Setup}}: For a field $k$, let $V, W$ be finite type $k$-group schemes, with $V$ smooth, arising as kernels in short exact sequences of finite type commutative $k$-groups
\[
0 \longrightarrow V \longrightarrow G_1 \xlongrightarrow{f} \overline{G}_1 \longrightarrow 0
\]
\[
0 \longrightarrow W \xlongrightarrow{l} G_2 \xlongrightarrow{g} \overline{G}_2 \longrightarrow 0
\]
Let $U$ be a central extension of $V$ by $W$. Suppose given a finite extension $k'/k$ such that we have an isomorphism $W_{k'} \times V_{k'} \xrightarrow{\phi'} U_{k'}$ of $W_{k'}$-torsors over $V_{k'}$.

Consider the following pushout diagram (the analogue of \ref{adiagram13}):
\begin{equation}
\label{adiagram213}
\begin{tikzcd}
1 \arrow{r} & W \arrow{r}{t} \arrow[d, hookrightarrow, "l"] & U \arrow{r}{\xi} \arrow[d, hookrightarrow, "j"] & V \arrow{r} \arrow[d, equals] & 1 \\
1 \arrow{r} & G_2 \arrow{r}{i} \arrow{d}{g} & M \arrow{r} \arrow{d}{\pi} & V \arrow{r} & 1 \\
& \overline{G}_2 \arrow[r, equals] & \overline{G}_2 &&
\end{tikzcd}
\end{equation}
and suppose also given an isomorphism $G_2 \times V \xrightarrow{\phi} M$ of $G_2$-torsors over $V$. Let $n: V_{k'} \rightarrow (\Ga^2)_{k'}$ denote the ``difference'' between the two isomorphisms $M_{k'} \simeq (G_2)_{k'} \times V_{k'}$ of $G_2$-torsors over $V_{k'}$ given by $\phi$ and the pushout $\eta'$ of $\phi'$ along the inclusion $l_{k'}: W_{k'} \hookrightarrow (G_2)_{k'}$. That is, $n = s - s'$ (difference defined via the $G_2$-action on $M$), where $s$ is given by the composition
\[
V_{k'} \xrightarrow{(0, {\rm{id}})} (G_2)_{k'} \times V_{k'} \xrightarrow{\phi_{k'}} M_{k'}
\]
and the map $s'$ is given by the same formula but with $\phi_{k'}$ replaced by the pushout $\eta'$ of $\phi'$.

\begin{proposition}
\label{connectingmap}
Suppose that the group law above on $M$ $($via the isomorphism $\phi$$)$ is given by
\begin{equation}
\label{grouplaw}
(\alpha_1, v_1) \cdot (\alpha_2, v_2) = (\alpha_1 + \alpha_2 + h(v_1, v_2), v_1v_2)
\end{equation}
for a bi-additive $h: V \times V \rightarrow G_2$. Further assume that, via the inclusion $V = \ker(f) \hookrightarrow G_1$, $h$ extends to a bi-additive map $G_1 \times G_1 \rightarrow G_2$, still denoted by $h$.

Let $\beta \in \overline{G}_1(k)$, and consider the exact sequence
\begin{equation}
\label{twist12}
1 \longrightarrow W \longrightarrow U_{\beta} \longrightarrow V \longrightarrow 1
\end{equation}
obtained by twisting the central extension $U$ of $V$ by $W$ by the cohomology class $\delta(\beta) \in {\rm{H}}^1(k, V)$, where $\delta \colon \overline{G}_1(k) \rightarrow {\rm{H}}^1(k, V)$ is the connecting map. Choose $\overrightarrow{X} \in G_1(k_s)$ such that $f(\overrightarrow{X}) = \beta$ $($possible because $V$ is smooth$)$. Let $S$ be a $k$-scheme, and let $v \in V(S)$. Then 
\[
g(h(\overrightarrow{X}, v)) - g(h(v, \overrightarrow{X})) + g(n(v)) \in \overline{G}_2(S),
\] 
and its image in $\overline{G}_2(S)/g(G_2(S)) \hookrightarrow {\rm{H}}^1(S, W)$ equals $\delta_{\beta}(v)$, where 
 $\delta_{\beta}: V(S) \rightarrow {\rm{H}}^1(S, W)$ is the connecting map associated to the twisted sequence $(\ref{twist12})$.
\end{proposition}

\begin{remark}
\label{oddassumptions}
We will apply Proposition \ref{connectingmap} to the groups $U_a, V_a, W_a$ constructed at the beginning of \S \ref{sectionpathologies1}, with $G_1 = G_2 = \Ga^2$ and $S = {\rm{Spec}}(k)$. Note that if $U = W \times V$ as $V$-schemes with group structure given by a bi-additive map $h: V \times V \rightarrow W$ that extends to a bi-additive map $h: \Ga^2 \times \Ga^2 \rightarrow \Ga^2$, then we may take this $h$ to be the $h$ in Proposition \ref{connectingmap}, $k' = k$, $\phi$ to be the pushout of $\phi'$, and $n = 0$. This holds in particular for the groups $U_a, V_a, W_a$ except when char$(k) = 2$ and $\F_4 \not \subset k$.
\end{remark}

\begin{proof}
Consider the pushout diagram (\ref{adiagram213}) with exact rows and columns. The group $V$ acts by ``conjugation'' on the whole diagram, so (strictly speaking, making use of the analogue of Remark \ref{twistingcocycle} to avoid ambiguities with a cocycle representing the class of $\beta$ in $\overline{G}_1(k)/f(G_1(k)) \hookrightarrow {\rm{H}}^1(k, V)$) we may twist by the class of $\beta \in \overline{G}_1(k)$ in $\overline{G}_1(k)/f(G_1(k)) \hookrightarrow {\rm{H}}^1(k, V)$ to obtain a twisted diagram
\begin{equation}
\label{diagram5}
\begin{tikzcd}
1 \arrow{r} & W \arrow{r} \arrow[d, hookrightarrow, "l"] & U_{\beta} \arrow{r}{\xi_{\beta}} \arrow[d, hookrightarrow] & V \arrow[d, equals] \arrow{r} & 1 \\
1 \arrow{r} & G_2 \arrow{r}{i_{\beta}} \arrow{d}{g} & M_{\beta} \arrow{r} \arrow{d}{\pi_{\beta}} & V \arrow{r} & 1 \\
& \overline{G}_2 \arrow[r, equals] & \overline{G}_2 &&
\end{tikzcd}
\end{equation}

We will see later (Lemma \ref{Galeqvtsection}) that there is a $k$-scheme-theoretic (which is not necessarily a homomorphism) section $V \rightarrow M_{\beta}$ to the map $M_{\beta} \rightarrow V$. In particular, one may lift elements of $V(S)$ to $M_{\beta}(S)$. We claim that for $v \in V(S)$, $\delta_{\beta}(v) \in \overline{G}_2(S)/g(G_2(S)) \hookrightarrow {\rm{H}}^1(S, W)$ is computed as follows. Lift $v$ to an element $m \in M_{\beta}(S)$. Then the class of $-\pi_{\beta}(m) \in \overline{G}_2(S)$ in $\overline{G}_2(S)/g(G_2(S)) \hookrightarrow {\rm{H}}^1(S, W)$ is $\delta_{\beta}(v)$. Indeed, the connecting map $V(S) \rightarrow {\rm{H}}^1(S, W)$ is defined by sending $v$ to the fiber $\xi_{\beta}^{-1}(v)$ above $v$ with its natural $W$-action, and similarly for the connecting map $\overline{G}_2(S) \rightarrow {\rm{H}}^1(S, W)$ in the first vertical sequence above. Thus, we need to construct an isomorphism $\epsilon: g^{-1}(-\pi_{\beta}(m)) \xrightarrow{\sim} \xi_{\beta}^{-1}(v)$ of $W$-torsors over $S$.

Given $x \in g^{-1}(-\pi_{\beta}(m))$ (more precisely, $x$ is an $R$-valued point of this fiber for some $S$-scheme $R$), consider $i_{\beta}(x)\cdot m \in M_{\beta}$. This maps to $0 \in \overline{G}_2$ under $\pi_{\beta}$, hence lies in $U_{\beta}$. Since $i_{\beta}(x) \mapsto 0 \in V$, and $m \mapsto v$, it follows that $u := i_{\beta}(x)\cdot m \mapsto v$, hence $u \in \xi_{\beta}^{-1}(v)$. This defines our map, and it is a straightforward diagram chase to see that this is a morphism of $W$-torsors and hence an isomorphism. So the main difficulty in computing $\delta_{\beta}$ explicitly is to compute a lift $m$ of $v$.

A $1$-cocycle valued in $V$ corresponding to $\beta$ is computed as follows. Let $\overrightarrow{X} \in G_1(k_s)$ be as in the statement of the proposition, so $f(\overrightarrow{X}) = \beta$. Then the image of $\beta$ in ${\rm{H}}^1(k, V)$ is represented by the Galois $1$-cocycle $\zeta\colon \sigma \mapsto {}^{\sigma}\overrightarrow{X} - \overrightarrow{X}$. The group $M_{\beta}$ is defined by Galois descent by taking $M_{k_s}$ and twisting the Galois action by this cocycle. More precisely, the new action is given by 
\begin{equation}
\label{newaction}
{}^{\sigma'}m = \zeta_{\sigma}\cdot {}^{\sigma}m \cdot \zeta_{\sigma}^{-1}
\end{equation}
for $m \in M(k_s)$, where the above (slightly abusive) notation means conjugation by $\zeta_{\sigma}$, which makes sense due to the centrality of $G_2$ in $M$. All of the maps in diagram (\ref{diagram5}) are the same as those in diagram (\ref{adiagram213}) over $k_s$, and are equivariant with respect to this new action, so give maps over $k$. 

The problem is that the obvious section $V_{k_s} \rightarrow (G_2)_{k_s} \times V_{k_s} \xrightarrow{\phi_{k_s}} M_{k_s}$ which is $0$ on the $G_2$-component is usually not Galois-equivariant with respect to this new action. So we need to construct a section that is Galois-equivariant. This is accomplished by the following lemma.

\begin{lemma}
\label{Galeqvtsection}
The map
\begin{equation}
\label{eqvtsection}
v \overset{\lambda}{\mapsto} \phi \left( h(v, \overrightarrow{X}) - h(\overrightarrow{X}, v), v\right) \in M_{k_s}
\end{equation}
is a Galois-equivariant section $($not necessarily a homomorphism$)$ to the map $M_{\beta} \rightarrow V$.
\end{lemma}

\begin{proof}
We will use the Galois-equivariant (because it is defined over $k$) map $\phi$ to identify $M$ with $G_2 \times V$ in order to ease notation. Then the group law on $M$ goes over to the group law (\ref{grouplaw}). We need to show that for $v \in V_{k_s}$, we have
\begin{equation}
\label{Galoiseqvt?}
\lambda({}^{\sigma}v) \overset{?}{=} {}^{\sigma'}\lambda(v).
\end{equation}

We first compute the left side of (\ref{Galoiseqvt?}). By definition of $\lambda$,
\begin{equation}
\label{Galoiseqvtleftside}
\lambda({}^{\sigma}v) = \left( h({}^{\sigma}v, \overrightarrow{X}) - h(\overrightarrow{X}, {}^{\sigma}v), {}^{\sigma}v\right).
\end{equation}
Now we compute the right side of (\ref{Galoiseqvt?}). In order to do this, we first compute the general formula for conjugation on $M$ via the group law formula (\ref{grouplaw}). So let $\alpha := (A, B)$, $m := (x, v) \in M$. Then, using the biadditivity of $h$, we have
\begin{align}
\label{conjugationformula}
\alpha \cdot m \cdot \alpha^{-1} & = (A, B)\cdot(x, v)\cdot (-A + h(B, B), -B) \\
&= (A + x + h(B, v), B + v)\cdot (-A + h(B, B), -B) = (x + h(B, v) - h(v, B), v).
\end{align}
In particular, as we noted earlier, the conjugation action only depends on the image $B$ of $\alpha$ in $V$. Using (\ref{newaction}), we have that the right side of (\ref{Galoiseqvt?}) equals
\[
\zeta_{\sigma}\cdot {}^{\sigma}\lambda(v) \cdot \zeta_{\sigma}^{-1}.
\]
Using (\ref{conjugationformula}) and the fact that $h$ is Galois-equivariant (because it is defined over $k$), we therefore find that the right side of (\ref{Galoiseqvt?}) has the same $V$-coordinate  ${}^{\sigma}v$ as the right side (see (\ref{Galoiseqvtleftside})), while its $G_2$-coordinate is given by the formula
\[
h({}^{\sigma}v, {}^{\sigma}\overrightarrow{X}) - h({}^{\sigma}\overrightarrow{X}, {}^{\sigma}v) + h({}^{\sigma}\overrightarrow{X} - \overrightarrow{X}, {}^{\sigma}v) - h({}^{\sigma}v, {}^{\sigma}\overrightarrow{X} - \overrightarrow{X}).
\]
Using the biadditivity of $h$, we see that this agrees with the $G_2$-coordinate of the left side (\ref{Galoiseqvtleftside}) of (\ref{Galoiseqvt?}). This completes the proof of the lemma.
\end{proof}

Returning to the proof of Proposition \ref{connectingmap}, by our discussion above the connecting map $\delta_{\beta}$ is obtained by negating $\pi_{\beta}$ of the right side of (\ref{eqvtsection}) for $v \in V(S)$ to get an element of $\overline{G}_2(S)$ representing a class in ${\rm{H}}^1(S, W)$. How do we compute $\pi_{\beta}$ in terms of the components on $M$ given by $\phi$? Recall that $\eta': (G_2)_{k'} \times V_{k'} \xrightarrow{\sim} M_{k'}$ is the isomorphism of $(G_2)_{k'}$-torsors over $V_{k'}$ obtained by pushing out $\phi'$ along the inclusion $l_{k'}: W_{k'} \hookrightarrow (G_2)_{k'}$. We know how to compute $\pi_{\beta}$ in terms of the components given by $\eta'$. Indeed, with respect to these components, $\pi_{\beta}$ is given simply by taking the $G_2$-component and applying $g$. In fact, $M$ is identified with $(G_2 \times U)/\psi(W)$, where $\psi$ is the anti-diagonal inclusion $w \mapsto (l(w), t(w)^{-1})$. Then $\pi_{\beta}$ is computed by lifting $m \in M$ to an element of $G_2 \times U$, projecting onto $G_2$, and applying $g$. 

But via the isomorphism $\eta'$, the $G_2$-component of a lift of $m' \in M_{k'}$ (more precisely, $m' \in M(R')$ for some $k'$-algebra $R'$) to $(G_2 \times U)_{k'}$ is identified with the $G_2$-component of $\eta'^{-1}(m')$, since $\eta'$ is the pushout of an isomorphism $\phi': W_{k'} \times V_{k'} \rightarrow U_{k'}$ of $W_{k'}$-torsors over $V_{k'}$. Thus, we need to apply $g$ to the $G_2$-component of $\eta'^{-1}$ applied to the right side of (\ref{eqvtsection}) (for $v \in V(S)$) and then negate the result. This new $G_2$-component is just the $G_2$-component of that right side minus $n(v)$ (since for $m \in M$ mapping to $v$, $n(v)$ is by definition the difference between the $G_2$-components of $\phi^{-1}(m)$ and $\eta'^{-1}(m)$):
\[
h(v, \overrightarrow{X}) - h(\overrightarrow{X}, v) - n(v).
\]
Now we must apply $g$. Using the fact that $g$ is a homomorphism, we obtain:
\begin{equation}
\label{computegeqn3}
g(h(v, \overrightarrow{X})) - g(h(\overrightarrow{X}, v)) - g(n(v)).
\end{equation}
Note that this quantity must lie in $G_2(S)$, since it equals $\pi_{\beta}$ of a lift of $v \in V(S)$ to $M_{\beta}(S)$. Finally, the {\em negative} of (\ref{computegeqn3}) represents the class of $\delta_{\beta}(v) \in \overline{G}_2(S)/g(G_2(S)) \hookrightarrow {\rm{H}}^1(S, W)$. This completes the proof of the proposition.
\end{proof}

Now we return to Gabber's groups $U_a, V_a, W_a$. We drop the $a$ subscript for notational convenience and denote these groups by $U, V, W$.

\begin{lemma}
\label{connectingmap_a}
Let $k$ be an imperfect field of characteristic $p$, and let $a \in k - k^p$. Let $U, V, W$ denote the groups $U_a, V_a, W_a$ constructed in $\S \ref{gabbergps}$. Let $f, g: \Ga^2 \rightarrow \Ga$ denote the maps $f(x, y) := x - x^{p^2} - ay^{p^2}$, $g(x, y) := x + x^p + ay^p$, so that $V = \ker(f)$ and $W = \ker(g)$. For $\beta \in k$ representing a class in  ${\rm{H}}^1(k, V) \simeq k/f(k^2)$, let $\delta_{\beta}: V(k) \rightarrow {\rm{H}}^1(k, W) \simeq k/g(k^2)$ denote the connecting map associated to the twisted sequence
\[
1 \longrightarrow W \longrightarrow U_{\beta} \longrightarrow V \longrightarrow 1.
\]
Then for $(c, d) \in V(k) \subset \Ga^2(k) = k^2$, the class $\delta_{\beta}(c, d) \in {\rm{H}}^1(k, W) \simeq k/g(k^2)$ is represented by:
\[
\begin{cases}
c^2\beta, & p = 2 \mbox{ and } \F_4 \subset k, \\
c^2\beta + c^3, & p = 2 \mbox{ and } \F_4 \not \subset k, \\
2c^p\beta, & p > 2.
\end{cases}
\]
\end{lemma}

\begin{proof}
We will use the notation from \S \ref{gabbergps}. We first treat the case $p > 2$. The group $U$ as a $W$-torsor over $V$ is $W \times V$, and the group structure is $(w, v) \cdot (w', v') = (w + w' + h(v, v'), v + v')$, where $h: V \times V \rightarrow W$ is the bi-additive map defined by $h((x, y), (x', y')) = (xx'^p - x^px', xy'^p - x'y^p)$. This of course extends to a bi-additive map $\Ga^2 \times \Ga^2 \rightarrow \Ga^2$ defined by the same formula, and by abuse of notation we still denote this map by $h$. By Remark \ref{oddassumptions}, we may apply Proposition \ref{connectingmap} with $k' = k$ and $n = 0$ to compute $\delta_{\beta}$.

Choose $(x, y) \in k_s^2$ such that 
\begin{equation}
\label{condeqn1}
f(x, y) = x - x^{p^2} - ay^{p^2} = \beta.
\end{equation}
Applying Proposition \ref{connectingmap} (with $G_1 = G_2 = \Ga^2$, and $f$ and $g$ the defining equations for $V$ and $W$), a straightforward computation shows that $\delta_{\beta}(c, d) \in {\rm{H}}^1(k, W) \simeq k/g(k^2)$ is represented by the element of $k$ given by
\begin{eqnarray}\label{calceqn1}
&& g(h((x, y), (c, d))) - g(h((c, d), (x, y))) \notag \\
&=& 2\left(xc^p-x^pc+x^pc^{p^2}-x^{p^2}c^p+x^p(ad^{p^2}) - c^p(ay^{p^2})\right).
\end{eqnarray}
Since $(c, d) \in V(k)$, we have
\begin{equation}
\label{condeqn2}
c - c^{p^2} - ad^{p^2} = 0.
\end{equation}
Solving for $ay^{p^2}$ and $ad^{p^2}$ in (\ref{condeqn1}) and (\ref{condeqn2}), and substituting into (\ref{calceqn1}), a straightforward computation shows that the right side of (\ref{calceqn1}) equals $2c^p\beta$.

Now suppose that $p = 2$. The calculation when $\F_4 \subset k$ is substantially the same as the one above, since in this case the group $U$ is once again defined by a bi-additive map $h_{\zeta}: V \times V \rightarrow W$ that extends to a bi-additive map $h_{\zeta}: \Ga^2 \times \Ga^2 \rightarrow \Ga^2$. Namely, the map is $h_{\zeta}((x, y), (x', y')) = (\zeta^2xx'^2 + \zeta x^2x', \zeta^2xy'^2 + \zeta x'y^2)$ for a primitive cube root of unity $\zeta \in \F_4^{\times}$. We choose $(x, y) \in k_s^2$ such that
\begin{equation}
\label{condeqn3}
f(x, y) = x + x^4 + ay^4 = \beta.
\end{equation}
Then Proposition \ref{connectingmap} tells us that $\delta_{\beta}(c, d) \in {\rm{H}}^1(k, W) \simeq k/g(k^2)$ is represented by the element of $k$ given by
\begin{eqnarray}\label{calceqn3}
&& g(h((x, y), (c, d))) - g(h((c, d), (x, y))) \notag \\
&=& cx^2 + c^2x + c^2x^4 + c^4x^2 + x^2(ad^4) + c^2(ay^4), 
\end{eqnarray}
where we have used the equality $\zeta^2 = \zeta + 1$. Since $(c, d) \in V(k)$, we have
\begin{equation}
\label{condeqn4}
c + c^4 + ad^4 = 0.
\end{equation}
Using (\ref{condeqn3}) and (\ref{condeqn4}) to solve for $ay^4$ and $ad^4$, and then substituting back into (\ref{calceqn3}), a straightforward computation shows that the right side of (\ref{calceqn3}) equals $c^2\beta$.

It remains to treat the case when $p = 2$ and $\F_4 \not \subset k$. This is trickier because in this case $U$ is not defined directly by a bi-additive map $V \times V \rightarrow W$, but rather one defines $U_{k(\zeta)}$ by such a map, and then defines $U$ by Galois descent. We use the notation of Proposition \ref{connectingmap}. Let us begin by working over $k(\zeta)$. Then we have an isomorphism $W_{k(\zeta)} \times V_{k(\zeta)} \xrightarrow{\phi'} U_{k(\zeta)}$ of $W_{k(\zeta)}$-torsors over $V_{k(\zeta)}$, and correspondingly, we obtain an isomorphism $(\Ga^2)_{k(\zeta)} \times V_{k(\zeta)} \xrightarrow{\eta'} M_{k(\zeta)}$ of $\Ga^2$-torsors over $V_{k(\zeta)}$, in which the group structure on $M_{k(\zeta)}$ (via the map $\eta'$) is defined by the bi-additive map $h_{\zeta}: V_{k(\zeta)} \times V_{k(\zeta)} \rightarrow (\Ga^2)_{k(\zeta)}$ given by $h_{\zeta}(v, v') = h^+(\zeta^2v, v')$ for $h^+(v, v') = b(v+v') - b(v) - b(v')$ with $b: V \rightarrow W \subset \Ga^2$ defined by the formula $b(x, y) = (x^3, xy^2)$.

If $\sigma$ denotes the nontrivial element of Gal$(k(\zeta)/k)$, then (see \S \ref{gabbergps}) the Galois descent datum on $U^{\zeta}$ which defines $U$ is given by the isomorphism $[\sigma]: U^{\zeta} \xrightarrow{\sim} \sigma^*(U^{\zeta}) = U^{\zeta + 1}$ given by $(w, v) \mapsto (w + b(v), v)$. Thus, the Galois action on $U^{\zeta} = U_{k(\zeta)}$ obtained via base change from the group $U$ is ${}^{\sigma}(w, v) = ({}^{\sigma}w - b({}^{\sigma}v), {}^{\sigma}v)$. Correspondingly, the Galois action on $(\Ga^2)_{k(\zeta)} \times V_{k(\zeta)} \underset{\sim}{\xrightarrow{\phi'}} M_{k(\zeta)}$ is ${}^{\sigma}(\alpha, v) = ({}^{\sigma}\alpha - b({}^{\sigma}v), {}^{\sigma}v)$.

The scheme-theoretic section $v \mapsto \eta' (0, v)$ of $M \rightarrow V$ over $k(\zeta)$ is not Galois-equivariant. We need to find a Galois-equivariant one, which will then descend to a section $V \rightarrow M$ over $k$. One easily checks that the section $v \mapsto \phi'(-\zeta b(v), v)$ does the job (using that $\zeta^2 + 1 = \zeta$ since $p = 2$).

We now have two $\Ga^2$-torsor isomorphisms of $M_{k(\zeta)}$ with $\Ga^2 \times V_{k(\zeta)}$: $\phi'$ and the one coming from this new section. The formulas for changing coordinates between these two identifications are
\[
(\alpha, v)_{{\rm{new}}} = (\alpha - \zeta b(v), v)_{\phi'}, 
\]
\begin{equation}
\label{changeofcoords}
(\alpha, v)_{\phi'} = (\alpha + \zeta b(v), v)_{{\rm{new}}}.
\end{equation}
(In terms of the language of the setup preceding Proposition \ref{connectingmap}, the ``new'' coordinates are the $\phi$-coordinates.)
Thus, we compute the group law on $M_{k(\zeta)}$ in the new coordinates as follows:
\begin{align*}
(\alpha, v)_{{\rm{new}}}\cdot (\alpha', v')_{{\rm{new}}} & = (\alpha - \zeta b(v), v)_{\phi'}\cdot (\alpha' - \zeta b(v'), v')_{\phi'} \\
& = (\alpha + \alpha' - \zeta b(v) - \zeta b(v') + h_{\zeta}(v, v'), v + v')_{\phi'} \\
& = (\alpha + \alpha' + h_{\zeta}(v, v') + \zeta \left((b(v+v') - b(v) - b(v')\right), v + v')_{{\rm{new}}}.
\end{align*}
It follows that via $\phi: \Ga^2 \times V \xrightarrow{\sim} M$, the $k$-group law on $M$ transports over to the one on $\Ga^2 \times V$ given by the $2$-cocycle $h_{{\rm{new}}}(v, v'): V \times V \rightarrow \Ga^2$ defined by $h_{{\rm{new}}}(v, v') := h_{\zeta}(v, v') + \zeta \left(b(v+v') - b(v) - b(v')\right)$. Explicitly, the role of $\zeta$ ``cancels out'' because one easily computes that
\[
h_{{\rm{new}}}((x, y), (x', y')) = (xx'^2, xy'^2).
\]
This is a bi-additive map $V \times V \rightarrow \Ga^2$ that extends to a bi-additive map $\Ga^2 \times \Ga^2 \rightarrow \Ga^2$ via the same formula, so we may apply Proposition \ref{connectingmap} to compute $\delta_{\beta}$ using $h_{{\rm{new}}}$ as $h$ there. Equation (\ref{changeofcoords}) shows that in the notation of that Lemma, $n(v) = \zeta b(v)$.

We will first compute $g(n(v))$ for $v = (c, d) \in V(k)$, using that $g(x, y) = x + x^2 + ay^2$. 
Since $(c, d) \in V(k)$, we have 
\begin{equation}
\label{condeqn6}
c + c^4 + ad^4 = 0.
\end{equation}
Using the relation $\zeta^2 = 1 + \zeta$, we obtain
\begin{align*}
g(n(v)) & = g(\zeta b(v)) \\
& = g(\zeta c^3, \zeta cd^2) \\
& = (c^6 + ac^2d^4) + \zeta(c^3 + c^6 + ac^2d^4) \\
& = c^2(c^4 + ad^4) + \zeta c^2(c + c^4 + ad^4) \\
& = c^2(c) + \zeta c^2(0) \\
& = c^3, 
\end{align*}
where the penultimate equality uses (\ref{condeqn6}). To summarize:
\begin{equation}
\label{g(n(v))eqn}
g(n(v)) = c^3.
\end{equation}

Let $(x, y) \in k_s^2$ satisfy
\begin{equation}
\label{condeqn5}
f(x, y) = x + x^4 + ay^4 = \beta.
\end{equation}
Then by Proposition \ref{connectingmap}, $\delta_{\beta}(c, d) \in {\rm{H}}^1(k, W) \simeq k/g(k^2)$ is represented by the element
of $k$ given by 
$$g(h_{{\rm{new}}}((x, y), (c, d))) - g(h_{{\rm{new}}}((c, d), (x, y))) + g(n(v)) 
= g(xc^2, xd^2) - g(cx^2, cy^2) + c^3.$$
By explicit computation of the latter via the definition of $g$, the right side is equal to
\begin{equation}\label{eqn12}
cx^2 + c^2x^4 + c^2(ay^4) + xc^2 + x^2c^4 + x^2(ad^4) + c^3, 
\end{equation}
where we have used (\ref{g(n(v))eqn}) and that $p = 2$. Solving (\ref{condeqn5}) and (\ref{condeqn6}) for $ay^4$ and $ad^4$ respectively, and plugging into (\ref{eqn12}) shows that (\ref{eqn12}) equals $c^2\beta + c^3$.
\end{proof}

We are now ready to prove Theorem \ref{innerinvariancefails}. By Lemmas \ref{tautimesSha} and \ref{V(k)finite}, it is enough to find an inner form $U'$ of $U$ with arbitrarily small Tamagawa number. Suppose given $\beta \in k$ representing an element in ${\rm{H}}^1(k, V) \simeq k/f(k^2)$. Then we may ``twist by $\beta$'' to obtain the exact sequence
\[
1 \longrightarrow W \xlongrightarrow{j_{\beta}} U_{\beta} \xlongrightarrow{\pi_{\beta}} V \longrightarrow 1.
\]
Since $W$ is unipotent, $\widehat{W}(k) = 0$. By Lemma \ref{tamsinsequences}, therefore, we obtain
\[
\tau(U_{\beta}) \cdot \# \left( \frac{V(\A)}{\pi_{\beta}(U_{\beta}(\A))V(k)} \right) = \tau(W)\tau(V) \cdot \# \ker(\Sha(j_{\beta})).
\]

The right side is bounded above independently of $\beta$, so in order to make $\tau(U_{\beta})$ arbitrarily small, we need to show that for any $M > 0$, there exists $\beta \in k$ such that $\# (V(\A)/\pi_{\beta}(U_{\beta}(\A))V(k)) > M$. Since $V(k)$ is finite (Lemma \ref{V(k)finite}), in order to do this it suffices (changing $M$) to find $\beta \in k$ such that
\[
\# \left( \frac{V(\A)}{\pi_{\beta}(U_{\beta}(\A))} \right) > M.
\]
In order to do this, in turn, it suffices to show that for any finite set $S$ of places of $k$, there exists $\beta \in k$ such that for all $v \in S$, the map $U_{\beta}(k_v) \rightarrow V(k_v)$ is not surjective. 

For each $v \in S$, choose $(c_v, d_v) \in V(k_v) \subset \Ga^2(k_v) = k_v \times k_v$ with $c_v \neq 0$. This is possible because $V(k_v)$ is a positive-dimensional Lie group over $k_v$, hence infinite. Now the point $(c_v, d_v)$ does not lie in $\pi_{\beta}(U_{\beta}(k_v))$ if and only if $\delta_{\beta}(c_v, d_v) \in {\rm{H}}^1(k_v, W_v)$ is nonzero. By Lemma \ref{connectingmap_a}, this is equivalent to
\begin{equation}
\label{betanotineqn}
\beta \notin \begin{cases}
c_v^{-2}g(k_v^2), & p = 2 \mbox{ and } \F_4 \subset k_v, \\
c_v^{-2}g(k_v^2) + c_v, & p = 2 \mbox{ and } \F_4 \not \subset k_v, \\
(2c_v^p)^{-1}g(k_v^2), & p > 2.
\end{cases}
\end{equation}
The map $g: k_v^2 \rightarrow k_v$ is induced by a smooth algebraic map (because $W$ is smooth), hence it is open. Therefore, the subgroups $g(k_v^2) \subset k_v$ are open, hence closed. Weak approximation then provides $\beta \in k$ satisfying (\ref{betanotineqn}) for all $v \in S$ provided that we show that these groups are not all of $k_v$. That is, we need to show that $g: k_v^2 \rightarrow k_v$ is not surjective. This is equivalent to having ${\rm{H}}^1(k_v, W) \neq 0$, which follows from \cite[Prop.\,5.15]{rospic}. The proof of Theorem \ref{innerinvariancefails} is complete. \hfill \qed

\section{Inner twisting: positive results}
\label{innertwistpos}

In this section we will prove Theorems \ref{tautwistpowerofpthm} and \ref{tauboundedinnerthm}, thereby showing that despite the pathologies exhibited in Theorem \ref{innerinvariancefails}, the variation of Tamagawa numbers within the set of inner forms of a given group does exhibit some regularity. A key to proving both results is the following lemma.

\begin{lemma}
\label{centralsubgroup}
Let $G$ be a connected linear algebraic group over a field $k$. If the $k$-unipotent radical $\mathscr{R}_{u, k}(G)$ is wound and nontrivial, then $G$ contains a nontrivial smooth connected central unipotent $k$-subgroup.
\end{lemma}

\begin{remark}
The wound assumption is critical, as one sees by considering the group $G = \Gm \ltimes \Ga$, with the action given functorially by $t\cdot x = tx$. In fact, as we will see in the proof below, the key point is that wound unipotent groups admit no nontrivial torus action.
\end{remark}

\begin{proof}
Let $U := \mathscr{R}_{u, k}(G)$. Consider the descending central series $\mathscr{D}_iU$ of $U$, defined inductively by $\mathscr{D}_0U := U$, and for $n \geq 0$, $\mathscr{D}_{n+1}U := [U, \mathscr{D}_nU]$, the commutator group of $U$ and $\mathscr{D}_nU$. Since $U$ is unipotent, it is a nilpotent group; that is, $\mathscr{D}_iU = 0$ for $i$ sufficiently large. Let $n$ be the maximal nonnegative integer such that $\mathscr{D}_nU \neq 1$. Then $W := \mathscr{D}_nU$ is a nontrivial wound smooth connected central characteristic subgroup of $U$; this last adjective means that it is preserved by all automorphisms and that this remains true after extension on $k$.

Let $P := G/U$ be the maximal pseudo-reductive quotient of $G$, and let $\pi: G \rightarrow P$ denote the quotient map. We will first show that $W$ is central in $\pi^{-1}(\mathscr{D}P)$. The group $W$ is normal in $G$, since it is a characteristic subgroup of the normal subgroup $U$. Further, since $W$ is central in $U$, $\mathscr{D}P$ acts on $W$ by ``conjugation'', namely, by lifting to $\pi^{-1}(\mathscr{D}P)$ and then acting by conjugation. We need to show that this action is trivial. But $\mathscr{D}P$ is equal to its own derived group \cite[Prop.\,1.2.6]{cgp}, hence is generated by its $k$-tori \cite[Prop.\,A.2.11]{cgp}. It therefore suffices to show that any torus in $\mathscr{D}P$ acts trivially on the wound group $U$. But wound unipotent groups admit no nontrivial torus action \cite[Prop.\,B.4.4]{cgp}. Therefore, $U$ is central in $\pi^{-1}(\mathscr{D}P)$.

This centrality implies that the smooth connected commutative affine group $C := P/\mathscr{D}P$ acts on $U$ by conjugation. Letting $T \subset C$ be the maximal torus, $T$ acts trivially on $U$, again because tori cannot act nontrivially on wound unipotent groups. Therefore, the unipotent quotient $V := C/T$ acts on $U$. We need to show that $U$ contains some nontrivial smooth connected subgroup on which $V$ acts trivially.

Consider the unipotent group $H := V \ltimes U$. If $H$ is commutative, then $V$ acts trivially on $U$, and we are done. Otherwise, $\mathscr{D}_1H = \mathscr{D}H$ is nontrivial, and all of the groups $\mathscr{D}_iH$ ($i > 0$) are contained in the normal subgroup $U \unlhd H$. Since $H$ is unipotent, there is a maximal positive integer $m$ such that $\mathscr{D}_mH \neq 1$. Then $V$ acts trivially on $\mathscr{D}_mH$, so we are done.
\end{proof}

We also need the following simple lemma.

\begin{lemma}
\label{splitquottam}
Suppose that we have an exact sequence
\[
1 \longrightarrow U \longrightarrow G \longrightarrow H \longrightarrow 1
\]
of connected linear algebraic groups over a global function field $k$, with $U$ split unipotent. Then $\tau(G) = \tau(H)$.
\end{lemma}

\begin{proof}
We claim that we have
\begin{equation}
\label{taupfspliteqn}
\tau(G) = \tau(H)\tau(U).
\end{equation}
Assuming this, then applying the lemma when $G, H$ are themselves split unipotent, and using the fact that $\tau(\Ga) = 1$ \cite[Chap.\,I, \S 5.14, Example 1]{oesterle}, we deduce by induction that $\tau(U) = 1$ for all split unipotent groups $U$, hence returning to the general case above (where $G, H$ are not necessarily split unipotent), $\tau(G) = \tau(H)$.

In order to prove (\ref{taupfspliteqn}), we first note that $\widehat{U}(k) = 0$, since $U$ is unipotent. It therefore suffices by Lemma \ref{tamsinsequences} to show that $\Sha(U) = 1$ and that the map $G(\A) \rightarrow H(\A)$ is surjective. In fact, since $U$ is split unipotent, we have ${\rm{H}}^1(k, U) = {\rm{H}}^1(k_v, U) = 1$ for every place $v$ of $k$, from which both assertions follow, the second also requiring \cite[Prop.\,1.5]{rospred}.
\end{proof}

\begin{proof}[Proof of Theorems $\ref{tautwistpowerofpthm}$ and $\ref{tauboundedinnerthm}$]
We proceed by induction on dim$(G)$, the $0$-dimensional case being trivial. (Actually, by phrasing things in terms of a minimal counterexample, we don't need to worry about the $0$-dimensional case.) Suppose that $G$ contains a nontrivial normal split unipotent subgroup $U$. Letting $H := G/U$, then we have an exact sequence
\[
1 \longrightarrow U \longrightarrow G \longrightarrow H \longrightarrow 1.
\]
Then $G/Z_G$ acts on this sequence, hence given any cocycle $\alpha \in Z^1(k, G/Z_G)$, we may twist the above sequence by a cocycle representing $\alpha$ to get a new sequence:
\[
1 \longrightarrow U_{\alpha} \longrightarrow G_{\alpha} \longrightarrow H_{\alpha} \longrightarrow 1.
\]
The group $U_{\alpha}$ is still split unipotent, as this may be checked over $k_s$. By Lemma \ref{splitquottam}, $\tau(G) = \tau(H)$ and $\tau(G_{\alpha}) = \tau(H_{\alpha})$. The propositions therefore follow for $G$ by induction, since $H_{\alpha}$ is an inner twist of $H$. (It is the twist of $H$ by the image of $\alpha$ under the map ${\rm{H}}^1(k, G/Z_G) \rightarrow {\rm{H}}^1(k, H/Z_H)$.)

We may therefore assume that $G$ contains no nontrivial normal split unipotent subgroups. Since the maximal split unipotent $k$-subgroup of $\mathscr{R}_{u, k}(G)$ is preserved by all $k_s$-automorphisms, so by the Zariski density of $G(k_s)$ in the smooth group $G$ it is normal in $G$, it follows that $\mathscr{R}_{u, k}(G)$ is wound unipotent. If $\mathscr{R}_{u, k}(G) = 1$ then $G$ is pseudo-reductive and we are done by Theorem \ref{innerinvariance}. So it only remains to treat the case in which $\mathscr{R}_{u, k}(G)$ is wound and nontrivial. Then by Lemma \ref{centralsubgroup}, $G$ contains a nontrivial smooth connected central unipotent subgroup $U \subset G$. Let $H := G/U$. Then we have the exact sequence with central kernel
\begin{equation}
\label{exactseq6}
1 \longrightarrow U \xlongrightarrow{j} G \xlongrightarrow{\pi} H \longrightarrow 1.
\end{equation}

Given a cocycle $\alpha \in Z^1(k, G/Z_G)$, we may twist (\ref{exactseq6}) by $\alpha$ to obtain the new sequence
\[
1 \longrightarrow U \xlongrightarrow{j_{\alpha}} G_{\alpha} \xlongrightarrow{\pi_{\alpha}} H_{\alpha} \longrightarrow 1,
\]
where the group $U$ is unchanged since $G/Z_G$ acts trivially on it (because $U$ is central in $G$). By \cite[Ch.\,III, \S 5.3, Thm.]{oesterle}, we have
\begin{equation}
\label{tam=1}
\tau(G)\cdot \# \left(\frac{H(\A)}{\pi(G(\A))H(k)}\right) = \tau(U) \tau(H)\cdot \# \ker(\Sha(j)),
\end{equation}
\begin{equation}
\label{tam=2}
\tau(G_{\alpha}) \cdot \# \left(\frac{H_{\alpha}(\A)}{\pi_{\alpha}(G_{\alpha}(\A))H_{\alpha}(k)}\right) = \tau(U) \tau(H_{\alpha})\cdot \# \ker(\Sha(j_{\alpha})).
\end{equation}
(We have used that $\widehat{U}(k) = 0$, since $U$ is unipotent. The finiteness of the coset spaces of adelic points is part of the statement of the cited result in \cite[Chap.\,III]{oesterle}, given the already-known finiteness of the other quantities appearing.) Strictly speaking, to make this conclusion we first need to verify that the subgroup $\pi(G(\A)) \subset H(\A)$ is normal, and similarly with the twisted map on adelic points. This follows from the fact that this image is the kernel of the map $H(\A) \rightarrow {\rm{H}}^1(\A, U)$, which is a group homomorphism because $U \subset G$ is central (where we are using \cite[Chap.\,I, \S 5.6, Cor.\,2]{serre} and \cite[Prop.\,1.5]{rospred}).

Let us first prove Theorem \ref{tauboundedinnerthm} for $G$. Using (\ref{tam=2}), we see that
\begin{eqnarray*}
\tau(G_{\alpha}) \leq \tau(G_{\alpha}) \cdot \# \left(\frac{H_{\alpha}(\A)}{\pi_{\alpha}(G_{\alpha}(\A))H_{\alpha}(k)}\right) 
&=& \tau(U) \tau(H_{\alpha}) \cdot \# \ker(\Sha(j_{\alpha}))\\
&\leq& \tau(U) \cdot \# \Sha(U) \cdot \tau(H_{\alpha}).
\end{eqnarray*}
The quantity $\tau(H_{\alpha})$ is bounded above independently of $\alpha$ (but of course depending on $G$) by Theorem \ref{tauboundedinnerthm} for $H$, which holds by induction, so we are done.

In order to prove Theorem \ref{tautwistpowerofpthm} for $G$, comparing (\ref{tam=1}) and (\ref{tam=2}) and using the fact that the proposition holds for $H$ by induction, we see that it suffices to show that the quantities $\# \left(\frac{H(\A)}{\pi(G(\A))H(k)}\right)$, $\# \left(\frac{H_{\alpha}(\A)}{\pi_{\alpha}(G_{\alpha}(\A))H_{\alpha}(k)}\right)$, $\# \ker(\Sha(j))$, and $\# \ker(\Sha(j_{\alpha}))$ are powers of $p$. We will show this for the ``untwisted'' quantities $\# \left(\frac{H(\A)}{\pi(G(\A))H(k)}\right)$ and $\# \ker(\Sha(j))$. The proofs for the twisted quantities are the same.

First, as we have already mentioned, the connecting map $H(\A) \rightarrow {\rm{H}}^1(\A, U)$ is a group homomorphism with kernel $\pi(G(\A))$. It follows that $H(\A)/\pi(G(\A))$ is a $p$-primary abelian group, since this holds for ${\rm{H}}^1(\A, U)$. Therefore, the finite quotient $H(\A)/\pi(G(\A))H(k)$ is a finite $p$-primary group, hence its order is a power of $p$. Similarly, in order to show that $\ker(\Sha(j)) \subset \Sha(U)$ has $p$-power order, it suffices to show that it is a subgroup. Since $U \subset G$ is central, we have an action of the group ${\rm{H}}^1(k, U)$ on the set ${\rm{H}}^1(k, G)$, which we will denote by $*$, such that the map ${\rm{H}}^1(k, U) \rightarrow {\rm{H}}^1(k, G)$ induced by the inclusion $U \hookrightarrow G$ is $\alpha \mapsto \alpha * 1$, where $1 \in {\rm{H}}^1(k, G)$ is the trivial element \cite[Chap.\,I, \S 5.7]{serre}. Therefore, $\ker(\Sha(j))$ is the intersection of $\Sha(U)$ with the stabilizer of $1 \in {\rm{H}}^1(k, G)$ for this action, hence it is a subgroup.
\end{proof}

\end{document}